\documentclass[12pt,a4paper,leqno]{amsart}
\advance\textwidth by 2.5cm %%3.6
\advance\oddsidemargin by -1.6cm \advance\evensidemargin by -1.6cm

\newcommand\cadlag{c{\'a}dl{\'a}g }

\newcommand\ps{$\mathfrak{P}=(\Omega,\mathcal{F},\mathbb{P})$}
\newcommand{\lb}{\langle}
\newcommand{\rb}{\rangle}
\newcommand{\embed}{\hookrightarrow}
\def\p{\parallel}

\newcommand{\Sub}{\mathrm{Sub}}
\newcommand{\LSub}{\mathrm{LSub}}
\newcommand{\cL}{{\mathcal L}}
\newcommand{\bounded}{{bnd} }
\newcommand{\unbounded}{{ubnd} }

\renewcommand{\Re}{\mathrm{Re}}
\newcommand{\PV}{\mathrm{PV}}

\newcommand{\drift}{b}

\def\todown{\searrow}

\usepackage{amssymb}
\input colordvi

\usepackage[OT2,OT1]{fontenc}
\newcommand\cyr{%
\renewcommand\rmdefault{wncyr}%
\renewcommand\sfdefault{wncyss}%
\renewcommand\encodingdefault{OT2}%
\normalfont \selectfont} \DeclareTextFontCommand{\textcyr}{\cyr}

\def\eps{\varepsilon}

\newcommand{\rf}[1]{(\ref{#1})}
\newcommand{\n}{\,\hbox{\vrule width 1pt height 8ptdepth 2pt}\,}

\theoremstyle{plain}
\newtheorem{theorem}{Theorem}[section]

\theoremstyle{remark}
\newtheorem{remark}[theorem]{Remark}
\newtheorem{example}{Example}[section]

\newtheorem{remarks}[theorem]{Remarks}

\theoremstyle{plain}
\newtheorem{corollary}{Corollary}[section]
\newtheorem{lemma}[theorem]{Lemma}
\newtheorem{proposition}[theorem]{Proposition}

\newtheorem{definition}[theorem]{Definition}

\begin{document}
\baselineskip 17.5pt %%normal
\numberwithin{equation}{section}

\title[OU process driven by a L{\'e}vy noise]
{Regularity of  Ornstein-Uhlenbeck processes  driven by a
L{\'e}vy white noise }

\author[Z Brze{\'z}niak]{Zdzis{\l}aw Brze{\'z}niak}
\address{Department of Mathematics\\
The University of York\\
Heslington, York YO10 5DD, UK} \email{zb500@york.ac.uk}

\author[J Zabczyk]{Jerzy Zabczyk}
\address{Institute of Mathematics, Polish Academy of Sciences, P-00-950 Warszawa, Poland}
\email{zabczyk@panim.impan.gov.pl}

\thanks{ \noindent \! \!  Supported by the Polish
Ministry of Science and Education project 1PO 3A 034 29 ``Stochastic
evolution equations with \! \! L\'evy noise''.}

%\keywords{}

%\subjclass[2000]{Primary: ; Secondary: }

\begin{abstract}
The paper is concerned with spatial and time  regularity of
solutions to linear stochastic evolution equation perturbed by
L{\'e}vy white noise "obtained by subordination of a Gaussian white
noise". Sufficient conditions for spatial continuity are derived. It
is also shown that solutions do not have in general \cadlag
modifications. General results are applied to equations with
fractional Laplacian. Applications to Burgers stochastic equations
are considered as well.
\end{abstract}

\date{\today}

\maketitle
%\tableofcontents

\section{Introduction}
\label{sec_intro}

The problem of spatial and time regularity of linear
stochastic equation with Gaussian noise is well understood, see e.g. \cite{DaPrato-Zab-1992}. The case
of L{\'e}vy type perturbations has not been investigated earlier
with an exception of the paper \cite{Lesc_R_2004} by Lescot and
R{\"o}ckner, treating stochastic heat equation with $\alpha$ stable
noise.

The present paper is devoted to a systematic study of the spatial
and time  regularity of the Ornstein-Uhlenbeck process with a
general L{\'e}vy white noise generalizing the so called stable white
noise. Let us recall that a $\mathbb{R}^d$ valued Wiener process
subordinated by a $\frac\alpha2$-stable, with $\alpha\in (0,2)$,
increasing process is a symmetric $\alpha$-stable process on
$\mathbb{R}^d$. It is therefore   natural to define $\alpha $-stable
white noise as a cylindrical Wiener processes on Hilbert spaces
subordinated by a real, $\frac\alpha2$-stable process, with
$\alpha\in (0,2)$. For an example see \cite{Lesc_R_2004}.  More
generally we define L{\'e}vy white noise as a cylindrical Wiener
process subordinated by an arbitrary real valued, increasing
L{\'e}vy process. It can be regarded as an alternative to
cylindrical L{\'e}vy processes introduced, for instance,  in
\cite{Peszat_Z_2007}, see also
\cite{Mueller,Mytnik,Peszat_Z_impulsive}. It is difficult to judge
at the moment which class will better suit the modelling purposes.

Let $Y$ be a L{\'e}vy white noise process on a Hilbert space $H$. Then the
Ornstein-Uhlenbeck process $X$ driven by $Y$ with generator $A$ is a
solution of the so called Langevin equation
\begin{equation}
\label{eqn_langevin_01}
\left\{\begin{array}{rcl}
dX(t)&=&AX(t)\,dt +dY(t),\; t\geq 0,\\
X(0)&=&x,
\end{array}\right.
\end{equation}
where $x\in E$ and $A$ is assumed to be an infinitesimal generator
of a $C_0$ semigroup $\textsf{S}=\big(S(t)\big)_{t\geq 0}$, on a Banach space $E$. In the most interesting cases the semigroup $\textsf{S}$ is  analytic and the space $E$ is  a subspace of the space $H$, i.e. $E\subset H$. Moreover,
\begin{equation}
\label{eqn_subordination}
Y(t)=W(Z(t)),\; t\geq 0,\\
\end{equation}
where $Z$ is a subordinator and $W$ is a cylindrical Wiener process
on $H$ that lives on some Banach space $U \supset H$. The latter
property is equivalent to the fact the embedding $H\embed U$ is
$\gamma$-radonifying, see e.g. \cite{Brzezniak_1995}. Our main aim
is to present conditions under which the solutions to
(\ref{eqn_langevin_01}) take values in the space $E$. This type of
regularity is of prime interest for the study of nonlinear
stochastic equations
\begin{equation}\label{eqn_nonlinear_01}
du(t)=(Au(t) + F(u(t))\,dt +dY(t),\; t\geq 0, \,u(0)=x,
\end{equation}
where $F$ is a nonlinear transformation. In fact if $Y_{A}$ is a
solution to (\ref{eqn_langevin_01}), with $x=0$, then
(\ref{eqn_nonlinear_01}) can be equivalently written as a stochastic
integral equation with random coefficients determined by $Y_A$:
\begin{equation}\label{eqn_nonlinear_02}
u(t) = S(t) x + \int_{0}^{t} S(t-s) F((u(s)+ Y_{A}(s))ds,\,\,t\geq
0.
\end{equation}
Thus the spatial regularity of $Y_A$ determines how irregular the
nonlinearity $F$ can be.  Of some importance here is also the
time regularity of the process. As we will see,  the solution
to \ref{eqn_langevin_01} does not have, in general, locally bounded
trajectories and therefore, to treat nonlinear equations
\ref{eqn_nonlinear_02} it is also necessary  to determine
when the process $Y_A$ has  trajectories  integrable in some powers.

The paper is organized as follows.  In the Preliminaries we recall
some basic results on  subordination. The nonstandard feature of our
presentation is that  the subordinated process is a
cylindrical one. In the next section we gather results on stochastic
integrals of Banach valued functions with respect to Poison  random
measure. Our main result is Theorem 4.1 from Section 4, formulating
sufficient conditions on the semigroup $S$ and on the intensity of
the subordinator implying regularity. Various applications are
treated in Section 5. A typical example  is the following equation
\begin{equation}
\label{eqn_langevin_Delta''} \left\{\begin{array}{rcl}
dX(t)&=&-(-\Delta)^\gamma X(t)+dY(t),\; t\geq 0,\\
X(0)&=&0
\end{array}\right.
\end{equation}
were $\gamma$ is a positive constant, $H= L^2(\mathcal{O})$ with
$\mathcal{O}\subset \mathbb{R}^d$ and $(-\Delta)^\gamma$ is the
fractional power of the Laplace operator $\Delta$ with the Dirichlet
boundary conditions. In particular, if $\gamma =1$, we get
stochastic heat equation
\begin{equation}
\label{eqn_langevin_Delta}
\left\{\begin{array}{rcl}
dX(t)&=&\Delta X(t)+dY(t),\; t\geq 0,\\
X(0)&=&0,
\end{array}\right.
\end{equation}
Our paper strengthens considerably a recent result by Lescot and
R\"ockner, who proved in \cite{Lesc_R_2004}, by employing Dirichlet
forms and Sazonov Theorem,  that if $d=1$, $\gamma=1$, and $Y$ is an
$\alpha$ stable process, $\alpha\in (0,2)$, then $X$ takes values in
the space $L^2(0,1)$, see also comments in subsection
\ref{subsec-LR}. In fact  a special case of our results reads as
follows, see Corollaries \ref{cor-Y^alpha_ii} and
\ref{cor-Y^alpha_i}.
\begin{theorem}\label{thm-intro} Suppose that $\alpha \in (0,2]$, $\gamma\in (0,2]$ and $\delta\in [0,\frac\gamma{\alpha\vee 1}-\frac{d}2)$.
Then the mild solution $u$ to problem \eqref{eqn_langevin_Delta''} is such that  for all $t\geq 0 $, with probability $1$,  $X(t)\in C_0^\delta(\mathcal{O})$.
\end{theorem}
The method of the proof is a generalization of a method from
\cite{Brzezniak_1997} (see also \cite[Appendix C]{Brz_Debbi_2007}
for details in the case $d=1$ and $\gamma<1$, where a purely
Gaussian case was considered). In that case, i.e. when in equation
\eqref{eqn_langevin_Delta''} we have purely Gaussian process, i.e.
an $L^2(0,1)$-cylindrical Wiener process, the mild solution $X$
takes values in the space $C_0^\delta(0,1)$, for
$\delta<\frac{\gamma}{2}-\frac12$. Note that here $d=1$ and
$\gamma=2$.

Some limits to the spatial regularity is the subject of Section 6.
 We consider the Langevin equation \eqref{eqn_langevin_Delta} with the
space $L^2(0,1)$  replaced by $C(\mathbb{S})$, where $\mathbb{S}$ is
the unit circle, and the operator $\Delta$ is replaced by the first
order operator $D=\frac{\partial}{\partial \sigma}$. We show that if
$Y = f L$ where $L$ is a real valued L{\'e}vy process with
trajectories of infinite variation then there exists $f\in
C(\mathbb{S})$ such that  the mild solution of the modified problem
\eqref{eqn_langevin_Delta} is not
$C(\mathbb{S})$-valued.  In this way we  generalize  a result by Dettweiler and van Neerven
\cite{Dettw_vN_2006} from the Gaussian to the L{\'e}vy case.\vspace{2mm}

In Section 7 we deal with the time regularity of the trajectories of
Ornstein-Uhlenbeck processes driven by L{\'e}vy white noise and show
that trajectories are not locally bounded but have some
integrability properties. In Section 8 we establish these properties
and apply them to  establish an existence result to Burgers equation
with L{\'e}vy white noise. We discuss only Burgers equation to fix
ideas, but we believe that other equations can be treated in a
similar way. The Burgers equation with L{\'e}vy type noise has been
considered in\cite{Truman}, \cite{Peszat_Z_2007} and recently in
\cite{Dong}. Our results are not directly comparable with that of
\cite{Truman} because we deal with different form of noise. The most
recent  paper \cite{Dong} considers compound Poisson noise and in
fact deals with deterministic Burgers equation on random intervals.
A possibility to generalize the results to non-autonomous equations
is discussed in the final Section 9.

\section{Preliminaries}
\label{sec_prelim}
\subsection{Subordination}
All processes considered in this paper will be defined on a fixed
probability space $\mathfrak{P}=(\Omega,\mathcal{F},\mathbb{P})$.
The law of a Banach space valued random variable $\xi$ will be
denoted by $\mathcal{L}(\xi)$. A family $(\mu_t)_{t\geq 0}$ of a
probability measure on a Banach space $U$ is a \textit{convolution
semigroup} if it it satisfies the following properties:
 \begin{trivlist}
\item[(i)] $\mu_0=\delta_0$,
\item[(ii)] $\mu_t\to \mu_0$ weakly as $t\todown 0$,
\item[(iii)] $\mu_t\ast\mu_s=\mu_{t+s}$, $t,s\geq 0$.
\end{trivlist}

By  definition, a \textit{subordinator process}, or  shortly
\textit{subordinator}, is an increasing L{\'e}vy process, see
\cite[Definitions 1.6 and 21.4]{Sato_1999}. We will need, see
\cite[Theorem 21.5]{Sato_1999}, the following fundamental
characterization of subordinator processes.

\begin{theorem} \label{thm-LK} Suppose that $Z=(Z(t))_{t\geq 0}$ is a subordinator process defined on the  probability space \ps. Then
there exists a real number $\drift  \in \mathbb{R}_+$, a non-negative measure $\rho$ on $(\mathbb{R}_+,\mathcal{B}(\mathbb{R}_+))$ satisfying
\begin{equation}\label{eqn-rho}
\rho(\{0\})=0,\;\int_1^\infty\rho(d\xi)+\int_0^1\xi\rho(d\xi)<\infty,
\end{equation}
 such that
\begin{equation}\label{eqn-LK}
\mathbb{E}\left( e^{-rZ(t)}\right)=e^{-t\tilde \psi(r)},\; r\geq 0,\, t\geq 0,
\end{equation}
where
\begin{equation}\label{eqn-psi}
\tilde \psi(r)=\drift  r+\int_0^\infty (1-e^{-r\xi})\rho(d\xi),\; r\geq 0.
\end{equation}
The family  $(\mu_t)_{t\geq 0}$, where $\mu_t$ is the laws of $Z(t)$, $t\geq 0$, forms a
convolution semigroup on $\mathbb{R}$.
\end{theorem}
The measure $\rho$, resp. the number $\drift $,  see discussions
preceding  \cite[Def. 11.9]{Sato_1999}, is called the intensity
measure, resp. the drift,  of the subordinator process $Z$.

\begin{definition}\label{example_2.1}
\begin{trivlist}
\item[(i)] For $p>0$ we will denote by $\Sub(p)$ the set consisting of all subordinator processes  $Z$ whose intensity measure $\rho$  satisfies
\begin{eqnarray}
\label{cond-2}
\int_0^1 \xi^{\frac{p}2}\rho(d\xi)&<&\infty.
\end{eqnarray}
It is obvious that if $0<p_1<p_2<2\leq p_3$, then $$\Sub(p_1) \varsubsetneq \Sub(p_2) \varsubsetneq \Sub(2) =\Sub(p_3).$$
\item[(ii)] Note that if $\beta \in (0,1)$ and   the measure $\rho$ is defined by
\begin{equation}\label{eqn-rho2}
\rho(d\xi)=\frac1{\beta\Gamma(1-\beta)\xi^{1+\beta}}1_{(0,\infty)}(\xi)\,d\xi,
\end{equation}
where $\Gamma$ is the Euler-gamma function  ($\Gamma(z)=\int_0^\infty t^{z-1}e^{-t}\,dt$, $\Re z>0$), then  the measure $\rho$
 satisfies  condition \eqref{eqn-rho}. If  $\drift =0$ and  $\beta \in (0,1)$,
then    $\tilde\psi=\tilde\psi_{\beta}=r^\beta$, $r\geq 0$. A
subordinator process corresponding to $\drift =0$ and the measure
$\rho$ defined by \eqref{eqn-rho2}    will usually be denoted by
$Z^\beta$. 
\end{trivlist}
\end{definition}

Assume now that   $(\zeta_s)_{s\geq 0}$ is a convolution semigroup
on a Banach space $U$ and $(\mu_s)_{s\geq 0}$ a convolution
semigroup on $\mathbb{R}_+$ corresponding to a subordinator. We define then the
subordinated laws on $U$ by
\begin{equation}\label{eqn-sub_law}
\tilde \zeta_t:= \int _0^\infty \zeta_s\,\mu_t(ds),\; t\geq 0.
\end{equation}

\noindent It is easy to check, that the subordinated laws form a
convolution semigroup as well:
$$
\tilde \zeta_{t+s}=\tilde \zeta_t\ast\tilde \zeta_s,\,\,\,t,s\geq 0.
$$
If $W$ and $Z$ are independent L{\'e}vy processes, with the
corresponding convolution semigroups $(\zeta_s)_{s\geq 0}$ and
$(\mu_s)_{s\geq 0}$, then the subordinated laws correspond to the
L{\'e}vy process $Y=(Y(t))_{t\geq 0}$ defined by \eqref{eqn_subordination}, i.e.
\begin{equation}\label{eqn-def-Y}
Y(t):=W(Z(t)),\; t\geq 0.
\end{equation}
We are interested in a special case of the general situation described in  the following result.
\begin{theorem}\label{thm-exist-levy}
Suppose that $H$ is a separable Hilbert space and $U$ is a separable
Banach space such that $H\subset U$ continuously and densely. Assume
that $Z$ is a   subordinator  process with  the  intensity measure $\rho$ and the
drift $\drift $. Assume that $W=(W(t))_{t\geq 0}$ is an $U$-valued Wiener process   with the RKHS of $W(1)$  equal to $H$ and let $\zeta_s=\mathcal{L}(W(s))$.
\begin{trivlist}
\item[\;\;i)]If the process $Y$ is defined by (\ref{eqn-def-Y}) then,
\begin{equation}\label{eqn-Y}
\mathbb{E}(e^{i \lb Y(t) ,\phi \rb_{U,U^\ast}})=e^{-t\tilde{\lambda}(\phi)},\; \phi\in U^\ast, \; t\geq 0 .
\end{equation}
where,  with   $\tilde{\psi}$ defined by \eqref{eqn-psi},
 \begin{equation}\label{eqn-tilde-lambda}
 \tilde{\lambda}(\phi) =\tilde{\psi}(\frac12 \vert \phi\vert_H^2), \phi\in H.
\end{equation}
\item[\;\;ii)] Moreover $Y$ is a $U$-valued L{\'e}vy process such that
\begin{equation}\label{eqn-Y-cf}
\mathbb{E}(e^{i\lb Y(t) ,\phi \rb})=e^{-t\, {\rm PV}\int_U\left(1-e^{i\lb u,\phi\rb}\right)\nu(du)},\; \phi\in U^\ast, \; t\geq 0,
\end{equation}
where the measure  $\nu$  is given by formula
\begin{equation}\label{eqn-intensity-measure}
\nu(\Gamma)=\int_0^\infty \zeta_s(\Gamma)\rho(ds),\; \Gamma\in \mathcal{B}(U),
\end{equation}
and
$${\rm PV}\int_U\left(1-e^{i\lb
u,\phi\rb}\right)\nu(du):=\lim_{\eps\todown 0} \int_{u\in U:|u|\geq
\eps} \left(1-e^{i\lb u,\phi\rb}\right)\nu(du).
$$
\item[\;\;iii)] The process $Y$ is of finite variation iff $$\int_0^1
\big[\int_{B_U(0,1)}|u|_U\,\zeta_s(du) \big]\rho(ds)<\infty.$$
\end{trivlist}
The process $Y$ will be called an $H$-cylindrical L{\'e}vy process subordinated by a (subordinator) process $Z$.
%with the zero drift and zero Gaussian part.
\end{theorem}
In the identity \eqref{eqn-tilde-lambda},  $\vert \phi\vert_H$
denotes the norm of $\phi$ regarded as a functional on $H$ while in
the identity  \eqref{eqn-Y}, the bracket
$\lb\cdot,\cdot\rb_{U,U^\ast}$ is the duality pairing between $U$
and $U^\ast$. Note that the RHS of the equality \eqref{eqn-Y} makes
sense since we identify $H$ with its dual $H^\prime$ so that
$U^\ast\subset H^\prime \equiv H\subset U$. Several aspects of
Theorem \ref{thm-exist-levy} can be found in \cite{Sato_1999}, where
a finite dimensional case is studied, and \cite{Linde_1986}, but we
provide a sketch of a proof for the readers convenience.
\begin{proof}[Proof of Theorem \ref{thm-exist-levy}]  First let us observe that the process $Y$ is a well defined  $U$-valued \cadlag process.
  For simplicity we will assume  that the process $W$, resp. $Z$,   is defined on a probability space $(\Omega_1,\mathcal{F}_1,\mathbb{P}_1)$, resp.  $(\Omega_2,\mathcal{F}_2,\mathbb{P}_2)$ and $\Omega=\Omega_1\times\Omega_2$, $\mathcal{F}=\mathcal{F}_1\otimes\mathcal{F}_2$ and $\mathbb{P}=\mathbb{P}_1\otimes\mathbb{P}_2$. Then, for any $\phi\in H$, we have the following sequence of equalities
\begin{eqnarray*}
\mathbb{E}(e^{i\lb Y(t) ,\phi \rb}_{U,U^\ast}) &=&  \mathbb{E}_1\mathbb{E}_2 (e^{i\lb W(Z(t,\omega_2),\omega_1)  ,\phi \rb}_{U,U^\ast})\\
=\mathbb{E}_2 \left[ \mathbb{E}_1 (e^{i\lb W(Z(t,\omega_2),\omega_1)
,\phi \rb}_{U,U^\ast})\right] &=& \mathbb{E}_2  e^{- Z(t,\omega_2)
\frac12 \vert \phi\vert^2}= e^{- t \tilde \psi(   \frac12 \vert
\phi\vert^2)} = e^{- t \tilde \lambda(  \phi)}.
\end{eqnarray*}
This concludes the proof  of  the identity \eqref{eqn-Y}.

In the first step in proving the identity
\eqref{eqn-intensity-measure} we shall prove that the measure $\nu$
defined by the RHS of equality \eqref{eqn-Y} satisfies identity
\eqref{eqn-Y-cf}. By the formulae
(\ref{eqn-Y}-\ref{eqn-tilde-lambda}) it is enough to show that
$\phi\in U^\ast$, $\tilde{\psi}(|\phi|_H^2)=\int_U\left(1-e^{i\lb
u,\phi\rb}\right)\nu(du)$. We have, by \eqref{eqn-Y} and
\eqref{eqn-psi}, where for simplicity we assume first that $\drift
=0$,
\begin{eqnarray*}
\PV \int_U\left(1-e^{i\lb u,\phi\rb}\right)\nu(du) &=&  \int_0^\infty  \PV\int_U\left(1-e^{i\lb u,\phi\rb}\right)\zeta_s(du) \rho(ds)\\
&=& \int_0^\infty \left(1-e^{-\frac{s}2
|\phi|_H^2}\right)\rho(ds)=\tilde{\psi}(\frac{1}2 |\phi|_H^2).
\end{eqnarray*}
This together with the identity \eqref{eqn-Y-cf} concludes the
proofs of the first two parts. To prove the final part let us
recall, see e.g.\cite{Peszat_Z_2007}, that  an $U$-valued L{\'e}vy
process $Y$ with intensity measure $\nu$ is of finite variation iff
$\int_{B_U(0,1)}|u|_U\,\nu(du)<\infty$. Let now $Y$ and $\nu$ be as
in Theorem \ref{thm-exist-levy}. Since then
$$\int_{B_U(0,1)}|u|_U\,\zeta_s(du)\leq
\int_{B_U(0,1)}\zeta_s(du)\leq \int_{U}\zeta_s(du)=1$$ and
    \begin{eqnarray*}&&\int_{B_U(0,1)}|u|_U\,\nu(du)=\int_0^\infty \big[\int_{B_U(0,1)}|u|_U\,\zeta_s(du) \big]\rho(ds)
    =\int_0^1 \big[\int_{B_U(0,1)}|u|_U\,\zeta_s(du) \big]\rho(ds)\\
    &+&\int_1^\infty \big[\int_{B_U(0,1)}|u|_U\,\zeta_s(du) \big]\rho(ds)
    \leq \int_0^1 \big[\int_{B_U(0,1)}|u|_U\,\zeta_s(du) \big]\rho(ds)+\int_1^\infty \rho(ds),\end{eqnarray*}
     in view of condition \eqref{eqn-rho} we infer that $Y$ is of finite variation iff $$\int_0^1 \big[\int_{B_U(0,1)}|u|_U\,\zeta_s(du) \big]\rho(ds)<\infty.$$
    In particular, if $H=U=\mathbb{R}$, $Y$ is of finite variation iff $\int_0^1 s^{1/2}\rho(ds)<\infty$.
\end{proof}

\begin{remarks}\label{rem-thm-exist-levy} %{remarks-1}
\begin{trivlist}
\item[(1)] 
     Note   that  since $H$ is the RKHS of $W(1)$,   the embedding $H\embed U$ is $\gamma$-radonifying.
    \item[(2)] In view of the Fernique Theorem, $\zeta_s$ has finite second moment and moreover there exists $C>0$ such that for  all $s>0$, $\int_U|u|^2\zeta_s(du)\leq C s$.
\item[(3)]  Given a separable Hilbert space $H$ and a real number $p\in (0,\infty)$ we will denote by $\LSub(H,p)$, the class of all L{\'e}vy processes $Y$ of the form \eqref{eqn-def-Y}, where $W$ is $H$-cylindrical Wiener process and  $Z$ is an  independent subordinator process belonging to class $\Sub(p)$.
     \item[(4)] In the special case of the subordinator process $Z^{\frac\alpha2}$, with $\alpha \in (0,2)$, i.e. when $\rho$ is defined by formula \eqref{eqn-rho2} with $\beta=\frac\alpha2$, the process $Y$ constructed in Theorem \ref{thm-exist-levy} will be denoted by $Y^{\alpha}$ {and called the  $H$-cylindrical $\alpha$-stable process}. Note that in this case
 \begin{eqnarray*}\label{eqn-Y^alpha}
\mathbb{E}(e^{i\lb Y^\alpha(t) ,\phi \rb})&=&e^{-t{\lambda_\alpha}(\phi)},\; \phi\in U^\ast, \; t\geq 0,
\end{eqnarray*}
where  ${\lambda_\alpha}$  is defined by
 \begin{eqnarray*}
{\lambda_\alpha}(\phi) &=& (\frac12)^{\frac{\alpha}2} \vert \phi\vert_H^{\alpha}, \phi\in H.
\end{eqnarray*}
 \item[(5)] It follows from Definition \ref{example_2.1}(ii) that $Y^\alpha$ belongs to the class $\LSub(H,p)$ iff $\alpha<p$.
\item[(6)]
Let now the L{\'e}vy process $Y$ and its intensity measure $\nu$ be as in Theorem \ref{thm-exist-levy}. By   part (iii) of that result,  the process $Y$  is of finite variation iff $\int_{B_U(0,1)}|u|_U\,\nu(du)<\infty$.  Since then

\begin{eqnarray*}
\int_{B_U(0,1)}|u|_U\,\zeta_s(du)\leq \int_{B_U(0,1)}\zeta_s(du)\leq \int_{U}\zeta_s(du)=1
\end{eqnarray*}
 and
    \begin{eqnarray*}\int_{B_U(0,1)}|u|_U\,\nu(du)&=&\int_0^\infty \big[\int_{B_U(0,1)}|u|_U\,\zeta_s(du) \big]\rho(ds)
    \\=\int_0^1 \big[\int_{B_U(0,1)}|u|_U\,\zeta_s(du) \big]\rho(ds)
    &+&\int_1^\infty \big[\int_{B_U(0,1)}|u|_U\,\zeta_s(du) \big]\rho(ds)\\
    \leq \int_0^1 \big[\int_{B_U(0,1)}|u|_U\,\zeta_s(du) \big]\rho(ds)&+&\int_1^\infty \rho(ds),\end{eqnarray*}
     in view of condition \eqref{eqn-rho} we infer that $Y$ is of finite variation iff $$\int_0^1 \big[\int_{B_U(0,1)}|u|_U\,\zeta_s(du) \big]\rho(ds)<\infty.$$
    In particular, if $H=U=\mathbb{R}$, the process $Y$ is of finite variation iff $\int_0^1 s^{1/2}\rho(ds)<\infty$.\\
\item[(7)] It follows from part (5) of these remarks that if in addition $H=U=\mathbb{R}$, the process $Y^\alpha$ is of finite
variation in $U$ iff $\int_0^1 s^{1/2}s^{-1-\frac{\alpha}2}\,ds<\infty$, i.e. iff $\alpha \in (0,1)$.
 \end{trivlist}
\end{remarks}

\begin{remark}
 The jump times of the process $Y$ are the same as of the process $Z$.
\end{remark}

\subsection{Ornstein-Uhlenbeck processes}

The main object
of our studies are Ornstein-Uhlenbeck processes on a Banach space
$E$, driven by L{\'e}vy white noise $Y$, that is solutions of the
Langevin equation:
\begin{equation}
\label{eqn_langevin_01'} \left\{\begin{array}{rcl}
dX(t)&=&AX(t)+dY(t),\; t\geq 0,\\
X(0)&=&x\in E,
\end{array}\right.
\end{equation}
Here  $E$ is a  Banach space such that $E\subset H$, $A$ is  an
infinitesimal generator of a  $C_0$ semigroup on  $E$,   $x\in E$
and $Y=(Y(t))_{t\geq 0}$  a L{\'e}vy white noise defined by
\eqref{eqn-def-Y}. 

 By a mild solution to problem \eqref{eqn_langevin_01'} we
understand an adapted $E$-valued process $X=(X(t))_{t\geq 0}$, such
that for each $t\geq 0$, the following holds
\begin{equation}
\label{sol_mild_langevin_01'} X(t)=e^{tA}x+\int_0^t
e^{(t-s)A}dY(s),\; t\geq 0.
\end{equation}
As  in the theory od stochastic evolution equations driven by a
cylindrical Wiener process, see e.g. \cite{Brz_vN_2000} we can
define a notion of a weak solution. In the Hilbert space case the
following definition has  been proposed by \cite{Peszat_Z_2007}, see
Definition 9.11. A {\it weak  solution} of \eqref{eqn_langevin_01'}
is an adapted  $E-$valued stochastic process $\{X(t)\}_{t\in \geq
0}$ such that for all $x^\ast\in D(A^\ast)$ the function $s\mapsto
\lb X(s) , A^\ast x^\ast\rb$ is almost surely locally integrable and
\begin{equation}
\label{sol_weak_langevin_01'} \lb X(t),x^\ast\rb = \int^t_0 \lb
X(s), A^\ast x^\ast\rb \,ds + \lb Y(t),x^\ast \rb \qquad t\geq 0.
\end{equation}
Equivalence between various  notions of solutions  in well known,
see e.g. Proposition 4.2 in \cite{Brzezniak_1995} for the case of
strict and mild solutions of equations in  Banach spaces  driven by
a Wiener process and Theorem 9.15 in \cite{Peszat_Z_2007} for the
case of   weak and   mild solutions of   equations in Hilbert spaces
driven by a L{\'e}vy noise. In the current work we will not dwell on
this issue and we will only consider mild solutions.

The following result is a direct extension to Banach spaces
of a result in the Hilbert space case, due to Chojnowska-Michalik,
see \cite{Chojnowska_1987}.

\begin{theorem}\label{thm-ChM}
If $U$ is a Banach space, $A$ is an infinitesimal generator of a $C_0$-semigroup $S=(S(t))_{t\geq 0}$ on $U$  and $Y$ is a $U$-valued process satisfying
\begin{equation}\label{eqn_char_f}
\mathbb{E}(e^{i\lb Y(t) ,\phi \rb})=e^{-t\lambda(\phi)}, \; \phi\in U^\ast,
\end{equation}
then the solution $X$ to the following stochastic Langevin equation
\begin{equation}\label{eqn_lang}
\left\{\begin{array}{rl}
dX(t)&=A X(t)\,dt+B dY(t),\\
 X(0)&=0,
 \end{array}
 \right.
\end{equation}
satisfies
\begin{equation}\label{eqn_char_OU}
\mathbb{E}(e^{i\lb X(t) ,\phi \rb})=e^{-\int_0^t\lambda(B^\ast S^\ast(\sigma)\phi)\,d\sigma},\; \;\phi\in U^\ast.
\end{equation}
\end{theorem}

 \begin{remark}
Since $X(t)=\int_0^tS(t-\sigma)BdY(\sigma)$, equality \eqref{eqn_char_OU} can be equivalently rewritten in terms of the measure $\mu_t=\mathcal{L}(X_t)$ as follows
\begin{equation}
\label{eqn_char_OU2}
\hat{\mu_t}(\phi)=e^{-\int_0^t\lambda(B^\ast S^\ast(\sigma)\phi)\,d\sigma}, \;\phi\in U^\ast.
\end{equation}
\end{remark}

\begin{proof}[Proof of Theorem \ref{thm-ChM}] Denote $\Psi(\sigma):=S(t-\sigma)B$.
By  the assumption \eqref{eqn_char_f} and  using approximations of Riemann sums we have

\begin{eqnarray*}
\mathbb{E}e^{\lb \xi,\int_0^t\Psi(\sigma)\,dY(\sigma)\rb} &=& \lim_{N} \mathbb{E}e^{\lb \xi,\sum_k \Psi(\sigma_k)(Y(\sigma_{k+1})-Y(\sigma_{k}))\rb}
\\
=\lim_{N} \prod_k \mathbb{E}e^{\lb \Psi(\sigma_k)^\ast\xi, (Y(\sigma_{k+1})-Y(\sigma_{k}))\rb} &=&\lim_{N} \prod_k
e^{- (\sigma_{k+1}-\sigma_{k})\lambda(\Psi(\sigma_k)^\ast\xi) }\\
= e^{- \lim_{N} \sum_k(\sigma_{k+1}-\sigma_{k})\lambda(\Psi(\sigma_k)^\ast\xi) } &=& e^{-\int_0^t\lambda( \Psi^\ast(\sigma)\phi)\,d\sigma}.
\end{eqnarray*}
\end{proof}

\section{Integrals with respect to Poisson Random measures}\label{sec:integral}

We will need several results on stochastic
integration of Banach space valued functions with respect to Poisson
random measure.  Suppose that $\pi$ is a Poisson random measure on a
measure space $(G,\mathcal{G})$ with intensity $\nu_\pi$ over the
probability space \ps. Suppose also that $E$ is a separable Banach
space.
\begin{theorem}\label{lemma-2}
  Suppose that $p\in (0,1]$. Then for an arbitrary strongly measurable  function $f: G\to E$, belonging to  $ L^p( G, \mathcal{G} ; \nu_\pi;E)$, the following inequality holds
 \begin{equation}\label{eqn-lemma2}
 \mathbb{E}\vert \int f(x)\pi(dx)\vert^p_E \leq \int \vert f(x)\vert^p\nu_\pi(dx).
 \end{equation}
  \end{theorem}
\begin{proof}[Proof of Theorem \ref{lemma-2}] Let us fix $p\in (0,1]$. Since the space of finitely-valued functions  is dense in
$ L^p( G, \mathcal{G} ; \nu;E)$, see Lemma 1.1. in
\cite{DaPrato-Zab-1992},
%\cite{Chalk-06}, it is assumed that  $p\geq 1$.} e.g. Lemma 1.2.14
%in \cite{Chalk-06},
we may suppose that $f=\sum_{i}f_i1_{ B_i}$ with
$f_i\in E$,  and $B_i\in\mathcal{G}$ is
  finite family of pair-wise disjoint sets such that  $\nu_\pi(B_i)<\infty$.  Then
$$
  \int_G f (x)\pi(dx)=\sum_{i} \pi({B_i }) f_i.$$
Moreover, as $\pi(B_i)$ takes values in $\mathbb{N}$,
$\mathbb{E}|\pi( B_i)|^p\leq \mathbb{E}|\pi( B_i)|$ . Since    the
random variables $\pi({B_i })$ are independent,   we infer that
\begin{eqnarray*}
 \mathbb{E}\vert \int_G \xi (x)\pi(dx)\vert_E^p&=&\mathbb{E}\big[\vert \sum_{i} \pi({B_i}) f_i \vert_E^p \big]
 \\
 &\leq&   \sum_{i} \vert f_i \vert_E^p \mathbb{E}|\pi( B_i)|^p \leq   \sum_{i} \vert f_i \vert_E^p \mathbb{E}|\pi( B_i)|
 \\&\leq &    \sum_{i} \vert f_i \vert_E^p  \nu_\pi(B_i)
 =  \mathbb{E}
\int_G \vert f(x)\vert_E^p \, \nu_\pi(dx).
 \end{eqnarray*}
 The proof is complete.
\end{proof}
 If $p\in (1,2]$, we will need that $E$ is o type $p$.  Let
us recall that a Banach space $E$ is is a type $p$, with $p\in
[1,2]$, iff there exists a constant $K_{p}(E)>0$ for any finite
sequence $\eps_1,\ldots,\eps_n \colon \Omega \to \{-1,1\}$ of
symmetric i.i.d. random variables and for any finite sequence
$x_1,\ldots,x_n$ of elements of $X$, the following inequality holds
\begin{equation}
 \mathbb{E} | \sum_{i=1}^n \eps_i x_i |^p \le
K_p(X) \sum_{i=1}^n | x_i |^p. \label{eqn-type_p}
\end{equation}
\begin{remark}
It is well known that all $L^q$ spaces with $q\geq p$ are type $p$
Banach spaces. Moreover, see e.g.
 Corollary A.6 in \cite{Brzezniak_1995}, the classical Sobolev $H^{s,q}$ and Besov-Slobodetski spaces $W^{s,q}$, with $q\geq p$,
 are also type $p$ Banach spaces.
\end{remark}
We have the the following extension of Lemma 7 from
\cite{Brz_H_2007p}.
\begin{theorem}\label{lemma-3}
 Suppose that  $p\in (1,2]$ and $E$ is a type $p$ Banach space.
 Then for an arbitrary  strongly measurable  function $f: G\to E$,  from $ L^p( G, \mathcal{G} ; \nu;E)$, the following inequality holds
 \begin{equation}\label{eqn-lemma2_b}
 \mathbb{E}\vert \int_G f(x)\tilde{\pi}(dx)\vert^p_E \leq 2^{2-p}  K_p(E) \int_G \vert f(x)\vert^p\nu_\pi(dx),
 \end{equation}
where $\tilde{\pi}$ is the compensated random measure defined by $\tilde{\pi}(B)=\pi(B)-\nu_\pi(B)$, $B\in\mathcal{G}$.
\end{theorem}

\begin{proof}[Proof of Theorem \ref{lemma-3}]The following equivalent characterization of type $p$ Banach spaces holds, see   \cite[Theorem
2.1]{HJ_P_1976}:

A    Banach space $X$ is is of type $p$ iff  there exists $K_p(X)$
such that for any finite sequence $\xi_1,\ldots,\xi_n $ of
 independent $X$-valued random variables  mean zero with  finite $p$-th moments,
\begin{equation}
 \mathbb{E} | \sum_{i=1}^n \xi_i  |^p \le
K_p(X) \sum_{i=1}^n  \mathbb{E} | \xi_i |^p. \label{eqn-2.3}
\end{equation}
As in the proof of Theorem \ref{lemma-2}  we may assume that $f=\sum_{i}f_i1_{ B_i}$  with
$f_i\in E$,  and $B_i\in\mathcal{G}$ is
  finite family of pair-wise disjoint sets such that  $\nu_\pi(B_i)<\infty$.  Since then
$$
  \int_G f (x)\tilde{\pi}(dx)=\sum_{i} \tilde{\pi}({B_i }) f_i$$
and    the random variables $\tilde{\pi}({B_i })$ are independent
and of  mean $0$, by the just quoted characterization of the type
$p$ property of the Banach space $E$ we infer that
\begin{eqnarray*}
 \mathbb{E}\vert \int_G \xi (x)\tilde{\pi}(dx)\vert_E^p&=&\mathbb{E}\big[\vert \sum_{i} \tilde{\pi}({B_i}) f_i \vert_E^p \big]
 \\
 &\leq&  K_p(E)  \sum_{i} \vert f_i \vert_E^p \mathbb{E}|\tilde{\pi}( B_i)|^p \leq K_p(E) 2^{2-p} \sum_{i} \vert f_i \vert_E^p \mathbb{E}|\tilde{\pi}( B_i)|
 \\&\leq &  K_p(E) 2^{2-p} \sum_{i} \vert f_i \vert_E^p  \nu_{\pi}(B_i)
 = K_p(E) 2^{2-p} \mathbb{E}
\int_G \vert f(x)\vert_E^p \, \nu_\pi(dx).
 \end{eqnarray*}
In the above we have used a fact, see Lemma 6 in \cite{Brz_H_2007p},
that if $\xi \sim \text{ Poiss }(\lambda)$, with $\lambda >0$ and
$p\in (1,2]$,  $\mathbb{E}|\xi-\lambda|^p \leq 2^{2-p}\lambda$.  The
proof is complete.
\end{proof}

\section{Space regularity of the Ornstein-Uhlenbeck process}\label{sec:O-U}

We  consider a triple of spaces $E$, $H$ and $U$ such that the middle one is a Hilbert space and the other two are separable Banach spaces and

\begin{equation}
\label{eqn_triple}
E \subset H \subset U
\end{equation}

We consider an $U$-valued Wiener process $W=\big(W(t)\big)_{t\geq 0}$ (with RKHS $H$) and a subordinator process $Z=\big(Z(t)\big)_{t\geq 0}$.
As in Theorem \ref{thm-exist-levy} we define an $U$-valued  L{\'e}vy process $Y=\big(Y(t)\big)_{t\geq 0}$ by the subordination formula \eqref{eqn_subordination}, i.e. \begin{equation}
\label{eqn_subordination'} Y(t):=W(Z(t)),\; t\geq 0.
\end{equation}

We also consider a $C_0$ semigroup $\textsf{S}=\big(S(t)\big)_{t\geq 0}$ of bounded linear operators on $E$. Our aim is to study spatial regularity of  Ornstein-Uhlenbeck processes  driven by the L{\'e}vy process $Y$ with generator $A$.  In other words we intend to study spatial regularity  of a process $X$ which  is a
solution of  the  Langevin equation \eqref{eqn_langevin_01}, i.e. 
\begin{equation}
\label{eqn_langevin_01_b}
\left\{\begin{array}{rcl}
dX(t)&=&AX(t)\,dt +dY(t),\; t\geq 0,\\
X(0)&=&0.
\end{array}\right.
\end{equation}

However we begin with  the following preliminary result.
\begin{theorem}\label{Theorem-1} Assume  that $T>0$, $p\in (0,2]$
and $\Psi(t)$, $t\in [0,T]$, is a strongly measurable function with values in the space of all  linear and bounded
operators from $U$ to $E$  such that
\begin{eqnarray}\label{cond-1}
\int_0^T\vert \Psi(s)\vert^p_{L(U,E)}\,ds&<&\infty.
\end{eqnarray}
Assume that $Z$ is a subordinator process belonging to the class
$\Sub(p)$. Let $Y$ be the $U$-valued L{\'e}vy process defined by
formula (\ref{eqn_subordination'}). Then, with probability $1$, $\int_0^t\Psi
(s)\,dY(s)$ takes values in $E$, for all $t\in [0,T]$, if
\begin{trivlist}
\item[(i)]  $p\in (0,1],$
\item[(ii)] or   $p\in (1,2]$ and the Banach space  $E$ is of   type
$p$.

\end{trivlist}
\end{theorem}

\begin{remark}\label{rem-Theorem-1}(i)
Let us notice that if the condition \eqref{cond-1} is satisfied and the embedding $i:H\embed U$ is $\gamma$-radonifying, then with $\Vert \cdot\Vert^p_{R(H,E)}$ denoting the $\gamma$-radonifying norm, see e.g. \cite{Brzezniak_1997},
\begin{eqnarray*}\label{cond-1'}
\int_0^T\Vert \Psi(s)\Vert^p_{R(H,E)}\,ds&<&\infty.
\end{eqnarray*}
The case when $E$ is a Hilbert space and $p=2$ is studied in \cite{Peszat_Z_2007}.\\
(ii) Let us notice that if the condition \eqref{cond-1} is satisfied and $p_2<p$, then it is also satisfied for $p_2$.\\
\end{remark}

\begin{proof}[Proof of Theorem \ref{Theorem-1}]
 The L{\'e}vy process  $Y$ can be  decomposed into two parts, the first one with small jumps and the second one with large jumps:
\begin{equation}\label{eqn-decom}
Y(t)=Y_1(t)+Y_2(t), \; t \geq 0
\end{equation}

The processes $Y_1$ and $Y_2$ are defined as follows.  Let  $\nu$ denotes the intensity measure corresponding to the process $Y$.
Then $Y_1$ is the L{\'e}vy process  corresponding  to the intensity measure $\nu_1$ defined by

\begin{equation}\label{eqn-nu1}
\nu_1(\Gamma):=\nu(\Gamma\cap B_U(0,1),
\end{equation}
 where $B_U(0,1)$ is the closed unit ball in $U$. We also  define $Y_2$ to be  the L{\'e}vy process  corresponding
 to the intensity measure $\nu_2=\nu-\nu_1$. Thus $Y_2$ is a compound Poisson process with the intensity measure
 $\nu_2$. The processes $Y_1$ and $Y_2$ can be easily constructed in
 terms of the  Poisson random measure $\pi$ associated with the process $Y$  defined by
\begin{equation}
\label{egn-jum-meas}
\pi([0,t]\times \Gamma):= \sum_{s\leq t} 1_{\Gamma}(\Delta Y(s)),\; \Gamma\in \mathcal{U},
\end{equation}
where $\Delta Y(s):=Y(s^+)-Y(s^-)$, $s\geq 0$. Note that we assume the process $Y$ to be right-continuous
with left hand limits, so that the definition of $\Delta Y$ makes sense and $Y(s^+)=Y(s)$.
It can be shown, see e.g. \cite[Prop. 15.5]{Kallenberg_2002} that $\pi$ is a time homogenous Poisson random measure and
 that $Y$ can be expresses in terms of the random measure $\pi$ as follows
$$
Y(t)=\sum_{s\leq t}\Delta Y(s)=\int_0^t\int_U u\pi(dy,ds)
$$

Then we define processes $Y_1$ and $Y_2$ by the following
modifications of the previous formula.
\begin{eqnarray}
\label{eqn-Y_2}
Y_2(t)&=&\sum_{s\leq t}1_{\{|\Delta Y(s)|\geq 1\}} \Delta Y(s)=\int_0^t\int_{|u|\geq 1} u\pi(dy,ds)\\
Y_1(t)&=&\sum_{s\leq t}1_{\{|\Delta Y(s)|< 1\}} \Delta Y(s)=\int_0^t\int_{|u|< 1} u\pi(dy,ds)
\label{eqn-Y_1}
\end{eqnarray}

Thus
\begin{equation}\label{eqn-Y1c}
\int_0^t \Psi(s)\,dY(s)= \int_0^t \Psi(s)\,dY_1(s)+
\int_0^t\Psi(s)\,dY_2(s),
\end{equation}
and we will check that the integrals with respect to $Y_1$ and $Y_2$
take values in $E$.

The process $Y_2$ has a discrete sequence of jump times:
$0<\tau_1<\tau_2<\tau_3<\cdots$ diverging to $+\infty$, and $\Delta
Y_2(\tau_k)=\Delta Y(\tau_k)$ for all $k$. Hence the integral with
respect to $Y_2$ is defined as
\begin{equation}
\label{int_Y_2} \int_0^t \Psi(s)dY_2(s):=\sum_{\tau_k\leq
t}\Psi(\tau_k)\Delta Y_2(\tau_k).
\end{equation}
Since the operators $\Psi$ take values in $E$  and the sum in \eqref{int_Y_2} is finite, it is clear that the
integral with respect to the process $Y_2$ takes values in $E$.

We pass now to the integral with respect to the process $Y_1$ and
assume that  $p\in (0, 1]$. Then,  by Lemma \ref{lemma-2},
 \begin{eqnarray*}
 \mathbb{E}
\vert \int_0^t\Psi(r)\, dY_1(r)  \vert_E^p &\leq&  \int_0^t\int_{B_U(0,1)} \vert \Psi(r)\vert_{\mathcal{L}(U,E)}^p  \vert u\vert_U^p \nu(du)\, dr\\
 &=& \int_0^t \vert \Psi(r)\vert_{\mathcal{L}(U,E)}^p \, dr \int_{B_U(0,1)}\vert u\vert_U^p \nu(du).
\end{eqnarray*}

It follows directly from formula \eqref{eqn-intensity-measure} in
Theorem \ref{thm-exist-levy} that for any $p\in (0,\infty)$,
 \begin{eqnarray*}
\int_{B_U(0,1)}\vert u\vert_U^p \nu(du)&=& \int_0^\infty \big[
\int_{B_U(0,1)} \vert u\vert_U^p \zeta_s(du) \big] \rho(ds)
\end{eqnarray*}
We have the following elementary lemma.

 \begin{lemma}\label{lem-5}
If $p\in (0, 2]$, then   or an arbitrary subordinator process $Z$ from the class $\Sub(p)$,
there exists a constant $C>0$ such that
 \begin{eqnarray}\label{s1}
 \int_1^\infty \big[ \int_{B_U(0,1)} \vert u\vert_U^p \zeta_s(du) \big] \rho(ds) &\leq& \int_1^\infty\, \rho(ds),
\\
\label{s2}
 \int_0^1  \int_{B_U(0,1)} \vert u\vert_U^p \zeta_s(du)  \rho(ds) &\leq&  C \int_0^1s^{\frac{p}2}\rho(ds),
\end{eqnarray}
where  $\nu$ is the  intensity measure of $Z$. In particular,
 \begin{eqnarray}
 \label{s3}
\int_{B_U(0,1)}\vert u\vert_U^p \nu(du)&\leq & \int_1^\infty\, \rho(ds)+C \int_0^1s^{\frac{p}2}\rho(ds).
\end{eqnarray}

\end{lemma}
\begin{proof}[Proof of Lemma \ref{lem-5}] To prove the first
inequality we note that
 \begin{eqnarray*}
 \int_1^\infty \big[ \int_{B_U(0,1)} \vert u\vert_U^p \zeta_s(du) \big] \rho(ds) &\leq&  \int_1^\infty \int_{B_U(0,1)} \vert u\vert_U^p \zeta_s(du)  \rho(ds)\\
\leq  \int_1^\infty \zeta_s(B_U(0,1)) \rho(ds) &\leq &
\int_1^\infty  \rho(ds) <\infty.
\end{eqnarray*}

 To prove the second
inequality we note that by the Cauchy-Schwartz inequality and part (3) of Remark \ref{rem-thm-exist-levy},
 \begin{eqnarray*}
\text{LHS} &\leq &  \int_0^1  \int_{U} \vert u\vert_U^p \zeta_s(du)  \rho(ds)\\
&\leq& \int_0^1  \big( \int_{U} \vert u\vert_U^2
\zeta_s(du)\big)^{\frac{p}2}  \rho(ds) \leq \int_0^1
(Cs)^\frac{p}2\rho(ds)=C\int_0^1 s^\frac{p}2\rho(ds).
\end{eqnarray*}
\end{proof}

By the above Lemma the case  $p\in (0,1]$ of Theorem \ref{Theorem-1} follows.

To treat the case of $p\in (1,2]$ note that he integral with respect
to $Y_1$ is defined as
\begin{eqnarray}
\nonumber
\int_0^t \Psi(s)\,dY_1(s)&:=&\int_0^t \int_{|u|<1}\Psi(s)u\pi(du,ds)\\&
= &\int_0^t \int_{|u|<1}\Psi(s)u\tilde{\pi}(du,ds)+\PV\int_0^t \int_{|u|<1}\Psi(s)u\nu(du)\,ds.
\label{int_Y_1}
\end{eqnarray}
Since $\PV\int_0^t \int_{|u|<1}\Psi(s)u\nu(du)\,ds=0$, see e.g. \cite[Thm 15.5]{Kallenberg_2002} we infer that
\begin{eqnarray}
\label{int_Y_1b}
\int_0^t \Psi(s)\,dY_1(s)&=&\int_0^t \int_{|u|<1}\Psi(s)u\,\tilde{\pi}(du,ds).
\end{eqnarray}

Since we assume that the space $E$ is martingale type $p$, by Lemma
\ref{lemma-3},  there exists a constant $C=C_p>0$ such that
$$\mathbb{E} \vert \int_0^t\int_{B_U(0,1)} \Psi(r)u
\tilde{\pi}(du,dr)  \vert_E^p \leq C
 \int_0^t\int_{B_U(0,1)} \vert \Psi(r)u \vert_E^p \nu(du)\,dr  .$$
 Therefore there exists a constant $C^\prime$ such that
 \begin{eqnarray}\nonumber
\mathbb{E} \vert \int_0^t\Psi(r)\, dY_1 (r) \vert_E^p &\leq  & C
 \int_0^t\int_{B_U(0,1)} \vert \Psi(r)u \vert_E^p \nu(du)\,dr\\
 \label{s4}
 &\leq& C^\prime  \int_0^t  \vert \Psi(r) \vert_{L(U,E)}^p  \, dr \int_{B_U(0,1)} \vert u \vert_E^p \nu(du)\,dr.
 \end{eqnarray}
Taking into account Lemma \ref{lem-5}, the result follows.
\end{proof}

Now we formulate the following important and useful consequence of Theorem \ref{Theorem-1}.

\begin{theorem}\label{Theorem-2} Assume  that  $p\in (0,2]$,   $E$ is a Banach space,  $Z$ is a subordinator process belonging to the class $\Sub(p)$ and  $Y$ is the L{\'e}vy process defined by formula (\ref{eqn_subordination'}).
 Assume  that  $\textsf{S}=\big(S(t)\big)_{t\geq 0}$ is $C_0$ semigroup on $E$ such that for each $t>0$, $S(t)$ is a 
 bounded  and linear operator from $U$ to $E$. Furthermore, we assume that  there exists $\theta >0$ with $\theta p <1$ such that and  for each $T>0$ one can find $C>0$ such that 
 \begin{equation}\label{ineq-A-1}
  \vert S(r)\vert_{\mathcal{L}(U,E)}\leq Cr^{-\theta},\; 0<r\leq T.
 \end{equation}

%\begin{eqnarray}\label{cond-1}
 Then for all $t\in \mathbb{R}_+$, with probability $1$, $\int_0^t S(t-r)\,dY(r)$ takes values in $E$ provided $p\in (0,1]$ or $p\in (1,2]$ and $E$ is of type $p$.
\end{theorem}
\begin{proof} It is enough to prove the Theorem for $t\in [0,T]$ for any fixed $T>0$. But then  the condition \eqref{ineq-A-1} implies  the condition \eqref{cond-1} is satisfied with $p<\frac1\theta$ by a family $\Psi(r)=S(T-r)$, $r\in (0,T]$.  Hence the result follows from \ref{Theorem-1}.
\end{proof}

\begin{remark}\label{rem-Theorem-2}{\rm
(i) As observed in the proof above,  condition \eqref{ineq-A-1} implies \eqref{cond-1} for $p<\frac1\theta$.  Hence we see the advantage of being able to treat the $p\leq 1$ case.\\
 (ii) If $\rho(ds)=s^{-1-\beta}$, then the condition \eqref{cond-2} is satisfied iff $p>\beta$. Hence, if $\beta$ is fixed there is a  lower limit on useful $p$.\\
 (iii) If $\rho(ds)=s^{-1-\beta}$, the condition \eqref{ineq-A-1} is satisfied and $\beta<\frac1\theta$, then  part (i) of Theorem \ref{Theorem-1} is applicable  with any $p\in (\beta,\frac1\theta)$. However, if $\beta\geq \frac1\theta$, then we need a different method. Such a  method is proposed in part (ii) of Theorem \ref{Theorem-1}.
(iv) If $E=U$  and
$\textsf{S}=\big(S(t)\big)_{t\geq 0}$ is a $C_0$ semigroup on $E$,
then the condition \eqref{cond-1} is satisfied for $p\in(0,2]$ and
hence Theorem above holds true. In particular, if $A$ denotes the
generator of the semigroup $\textsf{S}$, then, provided that $p\in
(0,1]$ or $p\in (1,2]$ and $E$ is a martingale type $p$ Banach
space,   for each $t>0$, $X(t)\in E$ a.s. where $X$ is
a solution of \eqref{eqn_langevin_01}. Moreover, if
$\textsf{S}=\big(S(t)\big)_{t\geq 0}$ is a $C_0$ \textit{group} on a
martingale type $p$ Banach space $E$, $p\in (1,2]$, then the
solution $X$ is $E$-valued \cadlag. Indeed, as in
\cite{Haus_Seidler_2008,Haus_Seidler_2001}, we have
$$X(t)=\int_0^tS(t-r)\, dY(r)=S(t)  \int_0^tS(-r)\, dY(r), \; t\geq 0.
$$
Since by \cite{Brz_H_2007p}, the process $\int_0^tS(-r)\, dY(r)$,
$t\geq$ is $E$-valued \cadlag, we infer that the process $X$ is also
$E$-valued \cadlag.

}\end{remark}

\section{Spatial regularity of  OU processes with fractional Laplacian }

In this section we investigate spatial
regularity of solutions $X$ of the Langevin equations
(\ref{eqn_langevin_Delta''}),
\begin{equation}\label{eqn_Lang_fractional}
\left\{\begin{array}{rcl}
dX(t)&+&(-\Delta)^\gamma X(t)=dY(t),\; t\geq 0,\\
X(0)&=&0,
\end{array}\right.
\end{equation}
where $(-\Delta)^\gamma$, $\gamma >0,$ is a fractional Laplacian,
with Dirichlet boundary conditions on a set $\mathcal{O}\subset
\mathbb{R}^d$  and $Y$ is a L{\'e}vy white noise in
$L^2(\mathcal{O})$. Our main aim is find conditions on the various
parameters guaranteeing that the solution $X$ takes values in the
space $E= C_0(\mathcal{O})$. We will treat the equation as an
abstract problem on a space $E$,
\begin{equation}
\label{eqn_langevin_01'} \left\{\begin{array}{rcl}
dX(t)&=&AX(t)+dY(t),\; t\geq 0,\\
X(0)&=&0,
\end{array}\right.
\end{equation}
with $E$ the Banach  space such that $E\subset H$, $A$ is  an
infinitesimal generator of a  $C_0$ semigroup on  $E$,   $x\in E$
and $Y=(Y(t))_{t\geq 0}$  a L{\'e}vy white noise defined by
\eqref{eqn-def-Y}.  We can  regard Laplacian $\Delta$,  on $\mathcal{O}$ and  with
zero boundary conditions,  as a self-adjoint operator and the
infinitesimal generator of a strongly continuous semigroup of
self-adjoint operators $T(t)$ on $L^{2}(\mathcal{O})$. The semigroup
can be restricted or extended to $L^{q}(\mathcal{O})$ in a natural
way. In this and similar situations the extensions or restrictions
will often be denoted in the same way as the original
objects.\vspace{3mm}

Our main tool in the study will be Theorem \ref{Theorem-1} and  the
validity of condition \eqref{ineq-A-1}. We will have also to check that
the space $U$ is chosen in such a way, that the process $Y$ is
$U$-valued, or equivalently that the embedding $H\embed U$ is
$\gamma$-radonifying.

\vspace{3mm}

\subsection{Analytic preliminaries}
Let us assume that  $\mathcal{O}\subset \mathbb{R}^d$ is a  bounded
open set with sufficiently smooth boundary.
 We will consider a real number $q$ such that $q\in(1,\infty)$.

  By
$H^{k,q}(\mathcal{O})$,  for $k \in \mathbb{N}$,  $q \in [1,\infty)$
we denote the Banach  space of all $f \in L^q(\mathcal{O})$  such
that  $D^\alpha f \in L^q(\mathcal{O})$, if the multi-index
$\alpha\in\mathbb{N}^d$ is of length $\leq  k$. The norm in
$H^{k,q}(\mathcal{O})$ is given by
\begin{equation}
\p f\p_{_{H^{k,q}(\mathcal{O})}}= \left( \sum_{|\alpha|\leq k} |D^\alpha
f|_{_{L^{q}(\mathcal{O})}}^q \right)^{\frac{d}q}. \label{2.1a}
\end{equation}
We define the fractional order Sobolev space
$H^{\beta,q}(\mathcal{O})$, $\beta \in \mathbb{R}_+\setminus
\mathbb{N}$ by the complex interpolation method, see
\cite{Lunardi_1995}, i.e.
\begin{equation}\label{eqn-sobolev_spaces}
H^{\beta,q}(\mathcal{O})=[H^{k,q}(\mathcal{O}), H^{m,q}(\mathcal{O})]_{\theta},
\end{equation}
 where
$k, m \in \mathbb{N}, \theta \in (0,1)$, $k<m$, are chosen to
satisfy
\begin{equation}\label{eqn-2.1c}
 \beta=(1-\theta)k+\theta m. \end{equation}
  One should bear in mind that the space on the LHS of formula
  (\ref{eqn-sobolev_spaces}) does not depend on $k,m,\theta$ provided they
  satisfy condition (\ref{eqn-2.1c}).

In what follows  by $H_0^{s,q}(\mathcal{O})$, $s\geq 0$,
$q\in(1,\infty)$ we will denote the closure of
$C_0^\infty(\mathcal{O})$ in the Banach space
$H^{s,q}(\mathcal{O})$. It is well known, see e.g. \cite[Theorem
11.1]{LM-72-i},     \cite[Theorem 1.4.3.2 p.317]{Triebel_1995}  and
   \cite[Theorem 3.40]{McLean-2000}  that
$H_0^{s,q}(\mathcal{O})=H^{s,q}(\mathcal{O})$ iff $s\leq \frac1{q}$.

\begin{remark}\label{rem-sob-spaces-type}
It is well known, see e.g. \cite{Pisier_1986} and \cite[Appendix B]{Brzezniak_1995} that all $L^q$ spaces, $q\geq p$, are of type $p$. By interpolation,
$H_0^{s,q}$ and $H_0^{s,q}$ spaces, $q\geq p$, are also of type $p$, see e.g. \cite[Appendix A]{Brzezniak_1995} for an argument used in the case of  martingale type $2$ spaces but which works also in the present situation.
\end{remark}
Let us denote by $A_q$ the Laplace operator in the space
$L^q(\mathcal{O})$, with the Dirichlet boundary conditions, i.e.
$D(A_q)=H^{2,q}\cap H_0^{1,q}$ and $A_qu=\Delta u$, for $u\in
D(A_q)$.

It is well known, see e.g. \cite{Triebel_1983}, that  $-A_q$ is an
operator with bounded imaginary powers and in particular the
operator $A_q^{\gamma/2}:=-(-A_q)^{\gamma/2}$, for $\gamma\in (0,2)$, is a generator of
an analytic semigroup on $L^q(\mathcal{O})$. Moreover,    it follows from Theorem 1.15.3 on p. 103 in
\cite{Triebel_1995} that for $\gamma \in (0,2)$,
$D(A_q^{\gamma/2})=D((-A_q)^{\gamma/2})=[L^q(\mathcal{O}),D(A_q)]_{{\gamma/2}}$, where
$[L^q(\mathcal{O}),D(A_q)]_{{\gamma/2}}$ is the complex interpolation
space of order ${\gamma/2}$, see e.g. \cite{LM-72-i},
\cite{Triebel_1995} and  \cite[Theorem 4.2]{Taylor_1981}. Moreover,
by Seeley \cite{Seeley_1972}, see also  the  monograph \cite[Theorem
in section 4.4.3]{Triebel_1995}, we have that for $q\in (1,\infty)$,

\begin{equation}\label{eqn-domains}
 D(A_q^{\gamma/2})=\begin{cases} H^{\gamma,q}(\mathcal{O})\cap H_0^{1,q}(\mathcal{O}), &
\text{ if } \; 1 < \gamma \leq 2 , \\
H^{\gamma,q}_0(\mathcal{O}), & \text{ if } \; \frac1{q}< \gamma \leq 1,\\
%\tilde{H}^{\gamma,2}(\mathcal{O}), & \text{ if } \gamma  = \frac12,\\
H^{\gamma,q}(\mathcal{O}), & \text{ if } \gamma  < \frac1{q}.
\end{cases}
\end{equation}

%\subsection{Probabilistic preliminaries}

\begin{proposition}\label{prop_ass1_Delta}  Assume that  $H=L^{2}(\mathcal{O})$ and $U=H^{-s,q}(\mathcal{O})$ with
 \begin{eqnarray}\label{ineq-HS}
  s&>&\frac{d}2, \;\; q\in [1,\infty).
   \end{eqnarray}
\begin{trivlist}
\item[\,\,i)] If
 \begin{eqnarray}\label{ineq-degree}
 r>0 ,\;  r+s<\gamma ,
   \end{eqnarray}
    then   the  $C_0$ semigroup  $\big(e^{tA_q^{\gamma/2}}\big)_{t\geq 0}$   satisfies  condition \eqref{ineq-A-1}
 with exponent $\theta:=\frac{r+s}\gamma$, $E=H_0^{r,q}(\mathcal{O})$ and  $U=H^{-s,q}(\mathcal{O})$,  
\item[\,\,ii)] If
 \begin{equation}\label{ass-boxed}
  \delta+\frac{d}q<\gamma-\frac{d}2,\; \delta \geq 0,
 \end{equation}
   then the semigroup $\big(e^{tA_q^{\gamma/2}}\big)_{t\geq 0}$ satisfies  condition \eqref{ineq-A-1}  with  $E=C_0^\delta (\mathcal{O})$, $U=H^{-s,q}(\mathcal{O})$ and the exponent $\theta\geq \frac{\delta+\frac{d}2}{\gamma}$. 
 \end{trivlist}
\end{proposition}

\begin{proof}[Proof of Proposition  \ref{prop_ass1_Delta}]
To prove the first statement  note that for $\gamma  < \frac1{q}$,
$H^{\gamma,q}_0(\mathcal{O})=H^{\gamma,q}(\mathcal{O})$. Let us
denote by $S=S_{q,\gamma}$, the semigroup on the Banach space
$L^q(\mathcal{O})$ generated by the operator $A_q^{\gamma/2}$, i.e. $S_{q,\gamma}(t)=e^{tA_q^{\gamma/2}}$, $t\geq 0$.
 It follows from \eqref{eqn-domains} and some standard estimates on the norm of the analytic semigroup between domains of fractional powers of the generator, see e.g. \cite{Pazy_1983}, that there exists a constant $C>0$ such that
 
\begin{eqnarray}\label{ineq-semigroup}
\vert S_{q,\gamma}(t)\vert_{{L}(H^{-s,q},H_0^{r,q})} &\leq & C \vert S_{q,\gamma}(t)\vert_{\mathcal{L}(D(A_q^{-\frac{s}\gamma}),D(A_q^{\frac{r}\gamma}))}\\
&\leq &
Ct^{-\frac{s}\gamma-\frac{r}\gamma}=Ct^{-\frac{r+s}\gamma}.\nonumber
\end{eqnarray}
Thus the first statement is a direct consequence of inequality
\eqref{ineq-semigroup}.

 To prove the second part of the Proposition
let us put $\eps:=\gamma-\frac{d}2-(\delta+\frac{d}q)$ which,  in
view of assumption \eqref{ass-boxed},  is $>0$. Let us choose a
numbers $r$  such that
  \begin{eqnarray}\label{eqn-4.2}
  r&>&\delta+\frac{d}q,
   \end{eqnarray}
 and $r-(\delta+\frac{d}q)<\frac{\eps}2$ and  a number $s$ satisfying the condition \eqref{ineq-HS} such that $s-\frac{d}2<\frac{\eps}2$. Then   $H_0^{r,q}(\mathcal{O})\subset E:=C_0^\delta (\mathcal{O})$ and  by  inequality \eqref{ineq-semigroup} we infer that
 \begin{eqnarray}\label{ineq-semigroup2}
\vert S_{q,\gamma}(t)\vert_{{L}(U,E)} &\leq & C t^{-\frac{r+s}\gamma}.
\end{eqnarray}
Since $r+s<(\delta+\frac{d}q)+\frac{\eps}2+\frac{d}2+\frac{\eps}2=\delta+\frac{d}q+\frac{d}q+\eps=\gamma$ we infer that   condition \eqref{ineq-A-1}  is satisfied with $\theta=\frac{\delta+\frac{d}2}{\gamma}$. It can easily shown that condition \eqref{ineq-A-1} is satisfied with any $\theta>\frac{\delta+\frac{d}2}{\gamma}$.
\end{proof}
\begin{remark}\label{rem-prop_ass1_Delta}
In fact we can replace the space $H=L^2(\mathcal{O})$ by $H^{\vartheta,2}(\mathcal{O})$ with, for concreteness, $\vartheta\in (-\frac12,\frac12)$. Let us recall that for $\theta\in [0,\frac12)$, $H^{\vartheta,2}(\mathcal{O})=H^{\vartheta,2}_0(\mathcal{O})$, i.e. the space $C_0^\infty(\mathcal{O})$ is dense in $H^{\vartheta,2}(\mathcal{O})$. Then an appropriate version of Proposition \ref{prop_ass1_Delta} holds true. Let us formulate, for future reference, the first two parts.\\
\textit{Let $H=H^{\vartheta,2}(\mathcal{O})$ with  $\vartheta\in (-\frac12,\frac12)$ and $U=H^{-s,q}(\mathcal{O})$.
\begin{trivlist}
\item[\,\,i)] If
 \begin{eqnarray}\label{ineq-HS'}
  s&>&\frac{d}2-\vartheta, \;\; q\in [1,\infty)
   \end{eqnarray}
then  the embedding $H\embed U$ is $\gamma$-radonifying.
\item[\,\,ii)] If \eqref{ineq-HS'} holds and
 \begin{eqnarray}\label{ineq-degree'}
 r>0 ,\;  r+s<\gamma
   \end{eqnarray}
    then,    the  family $\Psi=S_{q,\gamma}(\cdot)$ satisfies condition \eqref{ineq-A-1}  with exponent $\theta:=\frac{r+s}\gamma$ and    $E=H_0^{r,q}(\mathcal{O})$.
 \end{trivlist}
 }
Let us also note that the supremum of all numbers $r$ satisfying conditions (\ref{ineq-HS'}-\ref{ineq-degree'}) is equal to $\gamma+\vartheta-\frac{d}2$.
\end{remark}

\subsection{OU process driven by a cylindrical $\alpha$-stable processes with $\alpha<1$}
\label{subsec-4.1}
%Applications of part (i) of  Theorem \ref{Theorem-1}

The following result is  a direct consequence of Theorem
\ref{Theorem-1} part (i) and Proposition \ref{prop_ass1_Delta}. Let
us recall that by $\LSub(L^2,p)$ we denote the class of all
cylindrical L{\'e}vy processes defined by formula \eqref{eqn-def-Y}
with  $W$ being a cylindrical Wiener process on the Hilbert space
$H=L^2$ and $Z$ being a subordinator process belonging to the class
$\Sub(p)$.
\begin{theorem}\label{Theorem_appl_1_i}
Assume that $p\in (0,1]$ and that  $Y$ belongs to  $\LSub(L^2,p)$.
Assume that the degree $\gamma$ of the drift operator in equation \eqref{eqn_Lang_fractional} satisfies the assumption \eqref{ineq-degree}. Then, for each $s$ satisfying the condition \eqref{ineq-HS} (and the set of such numbers $s$ is non-empty), and  for all $q\in (1,\infty)$,  $r>0$  such that ${r+s}<\gamma$,
\begin{trivlist}
\item[\textit{(i)}]  The process $Y$ has an  $U=H^{-s,q}$-valued modification which is a L{\'e}vy process.
\item[\textit{(ii)}]
 For all $t\in [0,T]$, with probability $1$, the stochastic convolution $\int_0^t e^{(t-s)A_q^{\gamma/2}}\,dY(s)$ takes values in $H_0^{r,q}(\mathcal{O})$.
\end{trivlist}
Moreover, if $\delta\in [0,\gamma-\frac{d}2)$, then  appropriate $r$ and $s$ can be found such that

\begin{trivlist}
\item[\textit{(iii)}]
 For all $t\in [0,T]$, with probability $1$, the stochastic convolution $\int_0^t e^{(t-s)A_q^{\gamma/2}}\,dY(s)$ takes values in $C_0^\delta (\mathcal{O})$.
\end{trivlist}
\end{theorem}
\begin{proof}
Part (i) follows from  Theorem \ref {thm-exist-levy} and a well
known result for Gaussian,  see e.g. \cite{Brzezniak_1997}.

\end{proof}
\begin{remark}\label{rem-Theorem_appl_1_i}
(a) In the special case of the subordinator process
$Z^{\frac{\alpha}2}$ with  $\alpha \in (0,2)$, i.e. when $\rho$ is
defined by formula \eqref{eqn-rho} with $\beta=\frac\alpha2$, the
condition \eqref{cond-2} is satisfied when $\alpha<1$. Below we will
show how it is possible to remove the condition  $\alpha<1$ by
applying
part (ii) of Theorem \ref{Theorem-1}. \\
(b)
\end{remark}
In view of the  Remark \ref{rem-Theorem_appl_1_i} we have the following special case of Theorem   \ref{Theorem_appl_1_i}.

\begin{corollary}\label{cor-Y^alpha_i} Suppose that $\alpha \in (0,1)$,  and $\delta\in [0,\gamma-\frac{d}2)$.
Then,   for an appropriate $q\in(1,\infty)$, for all $t\in [0,T]$,
with probability $1$, the stochastic convolution $\int_0^t
e^{(t-s)A_q^{\gamma/2}}\,dY^\alpha(s)$ takes values in $C_0^\delta
(\mathcal{O})$.
\end{corollary}
\begin{remark}\label{rem-cor-Y^alpha_i}
It should be obvious that in Corollary \ref{cor-Y^alpha_i} the number $q$ plays only an auxiliary r\^ole. Firstly, we define the stochastic integral $\int_0^t e^{(t-s)A_q^{\gamma/2}}\,dY^\alpha(s)$ by means of an operator $A_q$ which depends $q$. Secondly, we prove that this stochastic integral takes values in the Sobolev space $H_0^{r,q}(\mathcal{O})$. Somehow, imprecisely, Corollary \ref{cor-Y^alpha_i} could be formulated as follows. \\
Suppose that $\alpha \in (0,1)$  and $\delta\in
[0,\gamma-\frac{d}2)$. Then  for all $t\in [0,T]$, with probability
$1$, the stochastic convolution $\int_0^t
e^{-(t-s)(-\Delta)^{\gamma/2}}\,dY^\alpha(s)$ takes values in
$C_0^\delta (\mathcal{O})$.
\end{remark}

\subsection{OU process driven by a cylindrical $\alpha$-stable processes with $\alpha>1$}
\label{subsec-4.2}
%{Applications of part (ii) Theorem \ref{Theorem-1}}

The previous argument works for a subordinator process satisfying the condition \eqref{cond-2} for some $p\in (0,1]$. In this section we consider  subordinator processes satisfying the condition \eqref{cond-2} for some $p\in (1,2)$.

%\subsection{Detailed analysis of the assumptions}

\begin{theorem}\label{Theorem_appl_1_ii}
Assume that $p\in (1,2)$ and that  $Y$ belongs to  $\LSub(L^2,p)$.
Assume that the degree $\gamma$ of the  drift operator in equation \eqref{eqn_Lang_fractional} satisfies the assumption
 \begin{eqnarray}\label{ineq-degree-p}
  \frac\gamma{p}&>&\frac{d}2
   \end{eqnarray}

 Then, for each $s$ satisfying the condition \eqref{ineq-HS}\footnote{Note and the set of such numbers $s$ is non-empty.} and  for all $q\in [p,\infty)$,  $r>0$  such that ${r+s}<\frac\gamma{p}$,
\begin{trivlist}
\item[\textit{(i)}]  The process $Y$ has a    $U=H^{-s,q}$-valued modification which is a L{\'e}vy process.
\item[\textit{(ii)}]
 For all $t\in [0,T]$, with probability $1$, the stochastic convolution $\int_0^t e^{(t-s)A_q^{\gamma/2}}\,dY(s)$ takes values in $H_0^{r,q}(\mathcal{O})$.
\end{trivlist}
Moreover, if
 \begin{eqnarray}\label{cond-3}
\delta\in [0,\frac\gamma{p}-\frac{d}2),
   \end{eqnarray}
 then  appropriate $r$ and $s$ can be found so that
\begin{trivlist}
\item[\textit{(iii)}]
 For all $t\in [0,T]$, with probability $1$, the stochastic convolution $\int_0^t e^{(t-s)A_q^{\gamma/2}}\,dY(s)$ takes values in $C_0^\delta (\mathcal{O})$.
\end{trivlist}
\end{theorem}

\begin{proof}[Proof of Theorem \ref{Theorem_appl_1_ii}] We only need to proof the part \textit{(iii)}. Trivially we can find $q\in [p,\infty)$ such
that $\delta +\frac{d}q+ \frac{d}2<\frac\gamma{p}$. In view of inequality \eqref{ineq-semigroup} in order to show that the
condition \eqref{cond-1} is satisfied, we need to show that we can find $r,s$ such that the conditions (\ref{eqn-4.2})-(\ref{ineq-HS})
hold true and $p\frac{r+s}\gamma<1$. Since the last condition is equivalent to $r+s<\frac\gamma{p}$, the proof is concluded by observing that with $\eps=\frac13\big(\frac\gamma{p}- \delta -\frac {d}q- \frac{d}2 \big)>0$, it is enough to put $r=\delta+\frac{d}q+\eps$ and $s=\frac{d}2+\eps$.
\end{proof}

%Theorem \ref{Theorem_appl_1_ii}, or rather Remark \ref{rem-Theorem_appl_1_ii}

\begin{remark}\label{rem-Theorem_appl_1_ii} As in Remark \ref{rem-prop_ass1_Delta} the Hilbert space $H=L^2$ can be replaced by $H^{\vartheta,2}(\mathcal{O})$ with, for concreteness $\vartheta\in (-\frac12,\frac12)$.  The appropriate version of \ref{Theorem_appl_1_ii} can be  formulated as follows. \\
\textit{Let $H=H^{\vartheta,2}(\mathcal{O})$ with  $\vartheta\in (-\frac12,\frac12)$ and $U=H^{-s,q}(\mathcal{O})$. Assume that $p\in (1,2)$ and that a cylindrical L{\'e}vy process $Y$ belongs to the class $\LSub(L^2,p)$, i.e. $Y$ is defined by formula \eqref{eqn-def-Y}, where
$W$ is a cylindrical Wiener process on the Hilbert space $H=H^{\vartheta,2}(\mathcal{O})$ with $\vartheta\in (-\frac12,\frac12)$, and $Z$ is a subordinator process belonging to the class $\Sub(p)$.
Assume that the degree $\gamma$ of the  drift operator in equation \eqref{eqn_Lang_fractional} satisfies the assumption
 \begin{eqnarray}\label{ineq-degree-p'}
  \frac\gamma{p}&>&\frac{d}2-\vartheta
   \end{eqnarray}
 Then, for each $s$ satisfying the condition \eqref{ineq-HS'} and  for all $q\in [p,\infty)$,  $r>0$  such that $r+s<\frac\gamma{p}$,
\begin{trivlist}
\item[\textit{(i)}]  The process $Y$ has a    $U=H^{-s,q}$-valued modification which is a L{\'e}vy process.
\item[\textit{(ii)}]
 For all $t\in [0,T]$, with probability $1$, the stochastic convolution $\int_0^t e^{(t-s)A_q^{\gamma/2}}\,dY(s)$ takes values in $H_0^{r,q}(\mathcal{O})$.
\end{trivlist}
 }
Let us also note that the supremum of all numbers $r$ satisfying the  conditions (\ref{ineq-HS'}) and (\ref{ineq-degree-p'}) is equal to $\frac\gamma{p}+\vartheta-\frac{d}2$.
\end{remark}

\begin{corollary}\label{cor-Y^alpha_ii} Suppose that $\alpha \in (0,2)$  and $\delta\in [0,\frac\gamma\alpha-\frac{d}2)$. Then there exists $q\in (\alpha\vee 1,\infty)$ such that $\delta +\frac {d}q+ \frac{d}2<\frac\gamma\alpha$ and
  for all $t\in [0,T]$, with probability $1$, the stochastic convolution $\int_0^t e^{(t-s)A_q^{\gamma/2}}\,dY^\alpha(s)$ takes values in $C_0^\delta (\mathcal{O})$.
\end{corollary}
\begin{proof}[Proof of Corollary \ref{cor-Y^alpha_ii}]

We take $\alpha \in (0,2)$ and put $\beta:=\frac\alpha 2$. Since the subordinator $Z^{\frac\alpha2}$ satisfies the condition \eqref{cond-1} with any $p>\alpha$ we can choose $p>\alpha$ such that the condition \eqref{cond-3} is satisfied. Then the result follows from Theorem \ref{Theorem_appl_1_ii}.
\end{proof}

\begin{remark}\label{rem-space_regularity} We infer from the above theorem that as $\alpha \in (0,2)$ decreases, the spatial regularity of the stochastic convolution $\int_0^t e^{(t-s)A_q^{\gamma/2}}\,dY^\alpha(s)$ increases.
\end{remark}
\begin{remark}\label{rem-cor-Y^alpha_ii}
It should be obvious that  in Corollary \ref{cor-Y^alpha_ii}  the number $q$ plays only an auxiliary r\^ole. Firstly, we define the stochastic integral $\int_0^t e^{(t-s)A_q^{\gamma/2}}\,dY^\alpha(s)$ by means of an operator $A_q$ which depends $q$. Secondly, we prove that this stochastic integral takes values in the Sobolev space $H_0^{r,q}(\mathcal{O})$. Somehow, imprecisely, Corollary \ref{cor-Y^alpha_ii} could be formulated as follows. \\
Suppose that $\alpha \in (0,2)$  and $\delta\in [0,\frac2\alpha-\frac{d}2)$.
Then  for all $t\in [0,T]$, with probability $1$, the stochastic convolution $\int_0^t e^{(t-s)\Delta}\,dY^\alpha(s)$ takes values in $C_0^\delta$.
\end{remark}

\subsection{Comparisons with Lescot-R\"ockner \cite{Lesc_R_2004}}
\label{subsec-LR}

%\begin{remark}
The case  of
equation \eqref{eqn_langevin_Delta} with $d=1$ and $Y$ such that
\begin{eqnarray*}\mathbb{E}(e^{i\lb Y(t) ,\phi \rb})&=&e^{-t\lambda(\phi)}, \;
 \lambda(\phi)=\int_0^1\phi(\sigma)^2\, d\sigma+ [\int_0^1\phi(\sigma)^2\,
 d\sigma]^{\frac\alpha2},\,\,\phi\in L^{2}(0,1),
\end{eqnarray*}
for some $\alpha \in (0,2)$,  has been studied in
\cite{Lesc_R_2004}. Note that by Theorem \ref{thm-exist-levy}, part
$i)$,
$$
\mathbb{E}(e^{i\lb X(t) ,\phi \rb})= e^{-N_1(\phi) - N_2(\phi)},
$$
where,
$$
N_1(\phi)= \int_0^t \int_0^1 |(S(s)\phi(\sigma)|^2 d\sigma ds
,\,\,N_2(\phi)= \int_0^t [\int_0^1 |(S(s)\phi(\sigma)|^2 d\sigma
]^{\alpha /2}ds,
$$
and $S$ is the heat semigroup on $L^{2}(0,1)$.

The paper \cite{Lesc_R_2004} relies on using the Sazonov topology
for the functionals $N_1$ and $N_2$ and  proves  that there exist
two Hilbert-Schmidt operators $U_1$ and $U_2$ on the space
$H=L^2(0,1)$ such that
\begin{eqnarray*}N_1(\phi)&=&\Vert U_1\phi\Vert, \;
N_2(\phi)\leq  \Vert U_2\phi\Vert^\alpha,\,\,\phi \in H.
\end{eqnarray*}
The above implies  that the functionals $N_1$ and $N_2$ are
continuous in the Sazonov topology so they are generating
probability measures on $H$. Hence the authors conclude that the
mild solution $X$ takes values in the space $L^2(0,1)$. We cover
more general case, i.e. we allow $d$ to be $\geq 1$ and consider
also fractional Laplacian. Our results in conjunction with the
purely Gaussian case as studied for example in \cite{Brzezniak_1997}
imply that also in the mixed case, the solution $X$ of equation
\eqref{eqn_langevin_Delta''} takes values in the space
$C_0^\delta(0,1)$, for $\delta<\frac{\gamma}{2}-\frac12$.

\subsection{Diagonal case} We begin with a
definition. If $b=(b_j)_{j\in\mathbb{N}}$ is a sequence of positive
numbers and $q\in [1,\infty)$, then by $l^q_b$ we mean a Banach
space  of all sequences $x=\big(x_j\big)_{j=1}^\infty$ such that
$\big(\sum_{j}b_j^{q}|x_j|^q\big)^{1/q}$ is finite. If
$a=(a_j)_{j\in\mathbb{N}}$ is a sequence of positive numbers then by
$a^{-1}$ we will understand a sequence
$a^{-1}:=(a_j^{-1})_{j\in\mathbb{N}}$.

Let us take two   increasing positive sequences $a=(a_j)_{j\in\mathbb{N}}$, $b=(b_j)_{j\in\mathbb{N}}$    convergent to $\infty$ and two numbers $r,q\in[1,\infty)$. Put
\begin{equation}\label{spaces-H,U,E}
H=l^2, \; U=l^r_{a^{-1}}, \; E=l^q_b.
\end{equation}
It  follows from Theorem 2.3 in \cite{Brz_vN_2003}, see also  \cite{Kwap_W_1992}, that  the embedding $H\embed U$ is $\gamma$-radonifying iff
\begin{equation}\label{eqn-diag-1}
\sum_{j}a_j^{-r}<\infty.
\end{equation}
Let $\lambda=(\lambda_j)_{j\in\mathbb{N}}$ be an  increasing positive sequence  convergent to $\infty$ and let $(e_j)_{j}$ be the standard ONB of $H$. Define a linear operator $A$ in $H$ by $D(A)=\{h\in H: \sum_j \lambda_j^2h_j^2 <\infty\}$, $Ah=\sum_j \lambda_jh_je_j$. Then  $A$ is a self-adjoint linear operator in $H$. Let $S(t)$ be the $C_0$ semigroup generated by the operator $-A$. Easy calculations show  that $S(t)$ has a unique extension to a bounded linear map from $U$ to $E$ and
\begin{equation}\label{eqn-diag-2}
\vert S(t)|_{L(U,E)}=\sup_{n}e^{-\lambda_nt}b_na_n^{r/q}.
\end{equation}
Consider now a special case, when for some $\alpha>0,\beta \geq 0$,
 $$a_n=\lambda_n^{\alpha}, \;b_n=\lambda_n^{\beta}.$$
 Since $ \sum_{j}a_j^{-r}=\sum_{j}\lambda_j^{-r\alpha}$
 the embedding $H\embed U$ is $\gamma$-radonifying iff
 $$\sum_{j}\lambda_j^{-r\alpha}<\infty.$$
 We also have
\begin{eqnarray*}\label{eqn-diag-2'}
\vert S(t)|_{L(U,E)}&=&\sup_{n}\big[e^{-\lambda_nt}\lambda_n^{\beta+\frac{r}{q}\alpha}\big]\\
&=& \sup_{n}\big[e^{-\lambda_nt}(t\lambda_n)^{\beta+\frac{r}{q}\alpha}\big]t^{-\beta-\frac{r}{q}\alpha}
\leq c^\ast t^{-\beta-\frac{r}{q}\alpha},
\end{eqnarray*}
where $c^\ast:=\sup_{x\geq 0}\big[e^{-x}(x)^{\beta+\frac{r}{q}\alpha}\big]=e^{-\beta-\frac{r}{q}\alpha}(\beta+\frac{r}{q}\alpha)^{\beta+\frac{r}{q}\alpha}$.

Thus we infer that if  $(\beta+\frac{r}{q}\alpha)p<1$, then

\begin{eqnarray}\label{eqn-diag-4}
\int_0^1 \vert S(t)|_{L(U,E)}^p\, dt &<&\infty.
\end{eqnarray}
 Hence we infer that for all choices of $\alpha$ and $\beta$ we can find $p\in (0,1)$ such that the condition \eqref{eqn-diag-4} holds.

Consequently, imposing some integrability conditions on the intensity of jumps of the subordinator process $Z$, i.e. $ \int_0^1s^{p/2}\rho(ds)<\infty$, one can show that the trajectories of the corresponding OU process belong to the space $E$.  These considerations can be summarized in the following result.

\begin{theorem}\label{theorem-3} Suppose that $p \in (0,2]$ and $Z$ is a subordinator process from the class $\Sub(p)$.
Suppose that $\lambda=(\lambda_j)_{j\in\mathbb{N}}$ is an  increasing positive sequence  convergent to $\infty$ and  $\alpha,\beta\in \mathbb{R}$ such that
$(\beta+\frac{r}{q}\alpha)p<1$ and $\sum_{j}\lambda_j^{-r\alpha}<\infty$. Put $a_n=\lambda_n^{\alpha}$, $b_n=\lambda_n^{\beta}$ and define the spaces $H,E,U$ by formula \eqref{spaces-H,U,E}, where in the case $p>1$ we assume additionally that $q\geq p$.  Let $W$ be a $U$-valued, $H$-cylindrical Wiener process and let $Y$ be a L{\'e}vy process defined  by $Y(t)=W(Z(t))$, $t\geq 0$. Then for all $t>0$, the stochastic convolution takes values in the space $E$.
\end{theorem}

\begin{remark}{\rm  The additional assumption that  $q\geq p$ in the case $p>1$ is made so that the space $E=l^q_{(b)}$ is of type $p$.
}\end{remark}

\noindent In a subcase when $r=q$ and $A_q$ denotes a linear
operator in $X=l^q$ defined by $D(A_q)=\{h\in l^q: \sum_j
\lambda_j^qh_j^q <\infty\}$, $Ah=\sum_j \lambda_jh_je_j$, the space
$E$ defined above equals to $D(A_q^{\beta})$.  Thus we have the
following corollary.

\begin{corollary}Suppose that $Z$ is a subordinator process such that for some $p \in (0,2]$, $ \int_0^1s^{p/2}\rho(ds)<\infty$ and $q\in(p\vee1,\infty)$.
Suppose that $\lambda=(\lambda_j)_{j\in\mathbb{N}}$ be an  increasing positive sequence  convergent to $\infty$ and  $\alpha,\beta\in \mathbb{R}$ such that
$(\beta+\alpha)p<1$ and $\sum_{j}\lambda_j^{-q\alpha}<\infty$. Let $H=l^2$ and $X=l^q$.  Let $W$ be a $D(A_q^{-\beta})$-valued, $H$-cylindrical Wiener process and define a L{\'e}vy process $Y$ by $Y(t)=W(Z(t))$, $t\geq 0$. Then for all $t>0$, with probability $1$, the stochastic convolution takes values in the space $D(A_q^{\beta})$.
\end{corollary}

\begin{remark} In all these considerations, the space $U$ plays only an auxiliary r\^ole.
\end{remark}
\begin{example} In the very special case when $\lambda_n=n^{2/d}$ for some $d>0$, for each $\beta\in (0,\frac1p)$ we can find $q$ and $\alpha >\frac{d}{2q}$
such that $\beta<\frac1p-\alpha$. In that case, for all $t>0$, with probability $1$, the stochastic convolution takes values in the space $D(A_q^{\beta})$.
\end{example}

\noindent We apply now the above method to the equation
(\ref{eqn_langevin_Delta''})
\begin{equation}
\left\{\begin{array}{rcl}
dX(t)&=&-(-\Delta)^\gamma X(t)+dY(t),\; t\geq 0,\\
X(0)&=&0
\end{array}\right.
\end{equation}
were $\gamma$ is a positive constant, $H= L^2(\mathcal{O})$ with
$\mathcal{O}=[0,\pi]^d$ and $\Delta$  is the Laplace operator with
zero Dirichlet boundary conditions. The Laplace operator and thus also
$(-\Delta)^\gamma$ are of diagonal type, with respect to
eigenfunctions
$$
e_j (\xi_1,\ldots,\xi_d) = (\sqrt\frac{2}{\pi})^d sin( n_1
\xi_1)\ldots sin (n_d \xi_d),
$$
where $j=(n_1,\ldots,n_d)$ is a multi-index of natural numbers. The
corresponding eigenvalues of the operator $-(-\Delta)^\gamma$ are
$$
\lambda_j = (n_1^2 +\ldots +n_d ^2)^\gamma .
$$
We set $\beta=0$, $q=1$ and   $b_j = 1$, $j\in\mathbb{N}$. Let
$E$ be the space of all uniformly
and absolutely convergent series
$$
\sum_{j} e_j (\xi) x_j,\,\, \xi \in ,; x=(x_j) \in l_{(b)}^{1}.
$$
Then $E$   can be identified with $l_{(b)}^{1}$ and   $E\subset
C_{0} ([0,\pi]^d)$ with the continuous embedding. Applying now
Theorem \ref{theorem-3}, with $a_j = (\lambda_j)^\alpha$,
$j\in\mathbb{N}$, we arrive at the following result.

\begin{theorem} Suppose that $p \in (0,{\frac{2\gamma}{d}}\wedge 1)$ and $Z$ is a subordinator process from the class $\Sub(p)$,  then
the solution of the equation (\ref{eqn_langevin_Delta''}) takes
values in the space $C_{0} ([0,\pi]^d)$.
\end{theorem}
\begin{proof} It is enough to find $\alpha$ and $r$ such that the assumptions of Theorem \ref{theorem-3} are satisfied. Firstly, we observe that he assumption that $\sum_{j}\lambda_j^{-r\alpha}<\infty$ is satisfied iff $r\alpha\frac{2\gamma}{d}>1$, i.e. $r\alpha<\frac{d}{2\gamma}$. Secondly, we observe that he assumption that $1>(\beta+\frac{r}{q})p$ is satisfied iff $r\alpha<\frac1p$. Since we assume that $\frac1p<\frac{d}{2\gamma}$, the desired numbers $r$ and $\alpha$ can be found and hence we get the result.
\end{proof}
\begin{remark}\label{rem-open question} Actually, the proof above would work also for $p \in (1,{\frac{2\gamma}{d}}\wedge 2)$ provided the space $E$ were of type $p$. We wonder if one could modify this proof so it also works for that range of $p$. However, we do not expect that such a proof would lead to a stronger result that Theorem \ref{Theorem_appl_1_ii}.
\end{remark}

\section{Limits to  spatial regularity}
It was shown in the preceding sections that the Ornstein-Uhlenbeck processes evolve in much smaller spaces than
the noise process. It seems that the reason for that phenomenon is the analyticity of the semigroup $S(t)$, $t\geq 0$. In this section we will show that, in
general, the Ornstein-Uhlenbeck processes do not  evolve even in the spaces in which the noise lives. An example of this sort, with the Wiener driving process,
was constructed by Dettweiler and van Neerven \cite{Dettw_vN_2006}.
 These authours have cleverly adapted an example from a joint paper of the current authours with Peszat \cite{Brz_P_Z_2001} and showed that there exists a Banach space $E=U$,  a $C_0$ semigroup on $E$ and an element $f\in E\setminus \{0\}$ such that the solution to the Langevin equation \eqref{eqn_langevin_01'}, i.e.
\begin{equation}
\label{eqn_langevin_01'''}
\left\{\begin{array}{rcl}
dX(t)&=&AX(t)+fdW(t),\; t\geq 0,\\
X(0)&=&0,
\end{array}\right.
\end{equation}
where $W$ is standard $\mathbb{R}$-valued Wiener process, does not take values in the space $E$.

We shall show below that   a generalization of the results from
\cite{Brz_P_Z_2001} published  recently in \cite{Kwap_M_R_2006}
leads to a result analogous to the one from \cite{Dettw_vN_2006}.
Namely, we will indicate a Banach space $E$, a $C_0$ group (with the
generator denoted by $A$) of linear bounded maps on $E$  and a
one-dimensional L{\'e}vy process $Y$ such that for some $f\in E$ the
mild solution to problem
\begin{equation}
\label{eqn_langevin_01_levy}
\left\{\begin{array}{rcl}
dX(t)&=&AX(t)+fdY(t),\; t\geq 0,\\
X(0)&=&0,
\end{array}\right.
\end{equation}
does not take values in $E$. We should point out here that our approach to the well-posedness is different to the one used in \cite{Dettw_vN_2006} and it could be used to give another proof of  the main result from \cite{Dettw_vN_2006}.

Let $\mathbb{ S} = \mathbb{R}/\!\!\mod 2\pi$  be the unit circle and  $E=C(\mathbb{ S})$ the Banach space of all continuous real-valued functions on $\mathbb{ S}$.
Note that $E$ can be identified, and this identification will be used below, with the space $\tilde E$ of all continuous $2\pi$-periodic functions from $\mathbb{R}\to \mathbb{R}$. Given $f\in E$ we will denote by $\tilde f$ the corresponding element of $\tilde E$. Let  $A = d/d\theta$ the infinitesimal generator of the rotation group $\textsf{S}=\{S(t)\}_{t\in \mathbb{R}}$ on $E$, defined via the above identification by $ \tilde{S}_t(\tilde f)=\tilde f(\cdot+t)$.

  \begin{theorem}\label{thm-non_exist} Assume that the real valued L{\'e}vy process $Y$ has trajectories with unbounded variation.
 There exists $f\in E$ such that a.s. the mild solution to problem \eqref{eqn_langevin_01_levy}
does not take values in  $E$.
 \end{theorem}
 \begin{proof}
 It follows from  the proof of Theorem 3.1, p. 599 in \cite{Kwap_M_R_2006}, that we can find $f\in E$ such that the  process 
 \begin{equation}
\label{eqn_diamond}
 \int_0^{2\pi} {\tilde f}(z-s ) dY(s), \quad z\in [0,2\pi],
 \end{equation}
has unbounded paths almost surely. On the other hand, since   $\tilde f$ is $2\pi$-periodic, the solution $X$ of the process \eqref{eqn_langevin_01_levy} satisfies
\begin{eqnarray*}
X(2 \pi)&=&\int_0^{2\pi}S(2\pi-s)f\,dY(s)= \int_0^{2\pi}\tilde f(\cdot+2\pi-s)\,dY(s)\\ &=& \int_0^{2\pi} \tilde f(\cdot-s)\,dY(s).
\end{eqnarray*}
Hence,
\begin{eqnarray*}
X(2 \pi,z)&=& \int_0^{2\pi} \tilde f(z-s)\,dY(s), \mbox{ for a.a. } z\in [0,2\pi].
\end{eqnarray*}
In this way the proof in complete.
  \end{proof}
\begin{remark} The proof of Theorem \ref{thm-non_exist} was based on showing that there exists $f\in E$ such that $X(2\pi)\notin E$ a.s. In this way, our result is a direct consequence of \cite{Kwap_M_R_2006}. We believe that as in \cite{Brz_P_Z_2001} one can show that
 set of all $f\in E$ such that  $X(2\pi)\notin E$ a.s.  is of 2nd category in $E$. \\
\end{remark}
 Note that Remark \ref{rem-Theorem-2}(iv) is applicable here but with a different choice of space $E$. We get the following result which, as it  should be noted,  does not contradict Theorem \ref{thm-non_exist}.
  \begin{theorem}\label{them_exist}  Assume that  $Z\in\Sub(p)$ with $p\in (1,2]$. Let  $E=W^{\theta,q}(\mathbb{S})$ with $q\in [p,\infty)$ and $\theta\geq 0$.
Then,  for each $f\in W^{\theta,q}(\mathbb{S})$, there exists a mild solution $X$ to the problem \eqref{eqn_langevin_01_levy} and $X$ is an $W^{\theta,q}(\mathbb{S})$-valued \cadlag process.\\
  In particular, if $\theta>\frac1q$,  $f\in W^{\theta,q}(\mathbb{S})$, the mild solution $X$ to the problem \eqref{eqn_langevin_01_levy} is a $C(\mathbb{S})$-valued \cadlag process.
 \end{theorem}
\begin{proof} The first part of the Theorem is a consequence of  Remark \ref{rem-Theorem-1}(iv) since the Besov-Slobodetskii space $W^{\theta,q}(\mathbb{S})$ is a Banach space of type $p$. The second part is a consequence of the first one and the the Sobolev embedding Theorem in view of which, under the assumptions of the Theorem,
 $W^{\theta,q}(\mathbb{S}) \embed C(\mathbb{S})$.
\end{proof}

\section{Time irregularity}

 We pass now to the time regularity and show  that in general
trajectories are not even locally bounded. We will work in the
framework of section \ref{sec:O-U}. In particular, we consider a
triple of spaces $E$, $H$ and $U$ such that the middle one is a
Hilbert space and the other two are Banach spaces and
\eqref{eqn_triple} hold.

We consider an $U$-valued Wiener process $W=\big(W(t)\big)_{t\geq 0}$ (with RKHS $H$), a subordinator process $Z=\big(Z(t)\big)_{t\geq 0}$ and
 an $U$-valued  L{\'e}vy proves $Y=\big(Y(t)\big)_{t\geq 0}$ by the subordination formula
 \eqref{eqn_subordination}, i.e.
 \begin{equation}
\label{eqn_subordination''} Y(t):=W(Z(t)),\; t\geq 0.
\end{equation}

Now we assume that $-A$ is a generator of an analytic $\textsf{S}=\big(e^{tA}\big)_{t\geq 0}$ semigroup in $E$ and $U$ and consider the stochastic convolution process
$$X(t)=\int_0^t e^{(t-s)A}dY(s),\;t\geq 0.$$ For simplicity, we only consider the case $p\in (1,2]$.
As in section \ref{sec:O-U} we decompose the process $Y$ into two independent parts, see equalities \eqref{eqn-Y_2}-\eqref{eqn-Y_1}, both being \cadlag,  and consider the processes $X_1$ and $X_2$ defined by
\begin{eqnarray}
\label{eqn-X_1}
X_1(t)=\int_0^t e^{(t-s)A}\, dY_1(s),\;t\geq 0,
\\
\label{eqn-X_2}
X_2(t)=\int_0^t e^{(t-s)A}\, dY_2(s),\;t\geq 0.
\end{eqnarray}

Our first aim is to prove the following generalization of   \cite[Theorem 9.2.4, p.269]{Peszat_Z_2007}.
\begin{theorem}\label{Thm-non-cadlag} In the framework of Theorem \ref{Theorem-2}, if  $F$ is a separable Banach space such that 
 \begin{trivlist}
\item[(i)] $E\subset F \subset U$ and  $\mu_1(F)=0$, i.e. $H$ is not embedded into $F$ in a $\gamma$-radonifying way,
\item[(ii)] $e^{tA}$ is a bounded linear map from  $U$ to $F$ for $t>0$ and 
 $$\{ x\in U: \sup_{t\in(0,1]}|e^{tA}x|_F<\infty\} \subset  F.$$
\end{trivlist}
Then   with positive probability, the process $\int_0^t e^{(t-s)A}\,dY(s)$, $t\geq 0$, is not  $F$-valued c{\'a}dl{\'a}g. \\
In particular, if a Banach space $F$ is such that $F\subset H$ and condition (ii) holds,  then   with positive probability, the process $\int_0^t S(t-r) (r)\,dY(r)$, $t \geq 0$, is not $F$-valued \cadlag.
\end{theorem}
\begin{remark}\label{rem-Thm-non-cadlag}
 \begin{trivlist}
\item[(i)] In view of Theorem \ref{Theorem-1},  the first part of condition (i) in Theorem \ref{Thm-non-cadlag} implies that $X$ is an $F$-valued process.
\item[(ii)] If  $-A$ is an infinitesimal generator of an analytic and \textit{compact} semigroup $\textsf{S}=\big(e^{tA}\big)_{t\geq 0}$  in  $U$ and $F=D(A^\beta)$, for $\beta>0$, with the graph norm,   then the condition (ii) in Theorem \ref{Thm-non-cadlag} is satisfied. Indeed, $A^\beta$ is a closed operator.
    \end{trivlist}

\end{remark}
\begin{proof} Let us fix $T>0$. Let $0<\tau_1<\tau_2<\tau_3<\cdots$
be the jump times of the process $Y_2$  and $\Delta
Y_2(\tau_k)=\Delta Y(\tau_k)- Y(\tau_k-)$, for all $k$. Since  $X_2(t)=0$ for $ t\in [0,\tau_{1})$ and $\tau_1$ is an exponential r.v. with mean value $>0$, we infer that $q_2:=\mathbb{P}(X_2 \mbox{ is bounded on } [0,T])$ is strictly positive. Assume for the time being that we  have shown that also $q_2<1$, i.e.
\begin{equation}\label{eqn-prob-unbnd}
1-q_2:=\mathbb{P}(X_2 \mbox{ is unbounded on } [0,T]) \mbox{ is strictly positive}.
\end{equation}
 Denote by $q_1$ the probability that  $X_1$ is is bounded on  $[0,T]$. Note finally that since $Y_1$ and $Y_2$ are independent processes so are $X_1$ and $X_2$. We have then the following simple chain of inequalities.
\begin{eqnarray*}
&&\mathbb{P}\big(X \mbox{ is \unbounded on } [0,T]\big) \geq  \mathbb{P}\big(X_1 \mbox{ is \bounded and } X_2  \mbox{ is \unbounded on } [0,T]  \\
&&\hspace{5truecm}\lefteqn{\mbox{ or } X_1 \mbox{ is \unbounded  and } X_2  \mbox{ is \bounded on } [0,T]\big)}\\
&=& \mathbb{P}\big(X_1 \mbox{ is \bounded  and } X_2  \mbox{ is \unbounded on } [0,T]\big)\\
&&\hspace{5truecm}\lefteqn{ +\,\mathbb{P}\big(X_1 \mbox{ is \unbounded  and } X_2  \mbox{ is \bounded on } [0,T]\big)}\\
 &=& \mathbb{P}\big(X_1 \mbox{ is \bounded  on } [0,T]\big) \, \mathbb{P}\big( X_2  \mbox{ is \unbounded on } [0,T]\big)\\
&&\hspace{3truecm}\lefteqn{+\mathbb{P}\big(X_1 \mbox{ is \unbounded  on } [0,T]\big)  \mathbb{P}\big( X_2  \mbox{ is \bounded on } [0,T]\big)}\\
 &=& q_1(1-q_2)+(1-q_1)q_2.
\end{eqnarray*}
Note that the number $q_1(1-q_2)+(1-q_1)q_2$ is a convex combination of numbers $q_2$ and $1-q_2$ and since both of them are strictly positive, so is $q_1(1-q_2)+(1-q_1)q_2$. This proves that $\mathbb{P}\big(X \mbox{ is \unbounded on } [0,T]\big)>0$ and hence the proof is complete. We only need to  prove claim  \eqref{eqn-prob-unbnd}. For  this aim let us notice that for $\omega\in\Omega$ and $t\in (\tau_1(\omega),\tau_2(\omega))$,
$$X_2(t,\omega)=e^{(t-\tau_1(\omega))A}\Delta Y(\tau_1(\omega))= e^{(t-\tau_1(\omega))A}\big( W(\beta\tau_1(\omega),\omega)-W(\beta\tau_1(\omega),\omega)\big).$$
Since with positive probability $\big( W(\beta\tau_1(\omega),\omega)-W(\beta\tau_1(\omega),\omega)\big)\notin F$, by (ii), we infer that $\limsup_{t\todown \tau_1}|X_2(t)|_F=\infty$.
\end{proof}

\section{Applications to nonlinear equations}
Let us consider, for a moment, nonlinear equation on a Banach space
$E$,
\begin{equation}
du(t)=(Au(t) + F(u(t))\,dt +dY(t),\; t\geq 0, \,u(0)=x,
\end{equation}
which can be rewritten as an integral equation
\begin{equation}\label{eq1}
u(t) = S(t) x + \int_{0}^{t} S(t-s) F((u(s)+ Y_{A}(s))ds,\,\,t\geq
0.
\end{equation}
where $Y_{A}$ is an Ornstein- Uhlenbeck process and thus solution to
the equation with $F=0$. As we already know, in general,  $Y_A$ has
locally unbounded trajectories and this  makes the equation
\ref{eq1} rather singular. However, if for instance  $F$ is
Lipschitz on $E$ then for the solvability of (\ref{eq1}) is enough
that  trajectories of $Y_A$ are in some $L^{p}([0,T], E)$.
\subsection{Time integrability of OU trajectories}
It turns out that in trajectories of the process $X$ have important
integrability properties. Taking into account the decomposition
(\ref{eqn-decom}) we write,
$$
X(t)= \int_0^t S(t-s) dY_1 (s)\,\,+\int_0^t S(t-s) dY_2 (s)= X_1(t)+
X_2(t),
$$
and from inequalities \eqref{s3} and \eqref{s4} we have the
following result.

\begin{theorem}\label{Thm-L^p-integrability}
If $p\in (1, 2]$,   $Z$ is an subordinator process from the class $\Sub(p)$, $E$ is a separable type $p$ Banach space,  then for some constant $C>0$,
\begin{equation}\label{s5}
\mathbb{E}\int_0^T|X_1(t)|_E^p \, dt\leq C T\int_0^T \vert e^{rA}\vert_{\mathcal{L}(U,E)}^p\, dr.
\end{equation}
Hence, if $\int_0^T \vert e^{rA}\vert_{\mathcal{L}(U,E)}^p\, dr<\infty$, then  $\mathbb{E}\int_0^T|X_1(t)|_E^p \, dt<\infty$ as well.\\
In particular, if for some $C>0$ and some $\theta\in (0,\frac1p)$,
 \begin{equation}\label{ineq-A-2} \vert e^{rA}\vert_{\mathcal{L}(U,E)} \leq Cr^{-\theta},\; r>0,
\end{equation}
then,  for some $C>0$ and all $T>0$,
\begin{equation}\label{s6}
\mathbb{E}\int_0^T|X_1(t)|_E^p \, dt\leq C T^{2-\theta p}.
\end{equation}
\end{theorem}
\begin{proof} By the Fubini Theorem we have the following chain of inequalities.
\begin{eqnarray*}
\mathbb{E}\int_0^T  |X_1(t)|_E^p \, dt&\leq&  C \int_0^T \int_0^t \vert e^{(t-s)A}\vert_{\mathcal{L}(U,E)}^p\, ds\, dt\\
=C \int_0^T \int_0^t \vert e^{rA}\vert_{\mathcal{L}(U,E)}^p\, dr\, dt &\leq &  \int_0^T \int_0^T \vert e^{rA}\vert_{\mathcal{L}(U,E)}^p\, dr\, dt
\\&=& C T\int_0^T \vert e^{rA}\vert_{\mathcal{L}(U,E)}^p\, dr.
\end{eqnarray*}
This proves inequality \eqref{s5} while the proof of \eqref{s6} follows the same lines.
\end{proof}
\begin{remark} Note that  assumption \eqref{ineq-A-2} is sufficient for both the existence of $E$-valued modification of $X$ and for the $p$-integrability of the $E$-norm of trajectories of the process $X$.
\end{remark}
\begin{corollary}\label{Cor-L^p-integrability}
Under the assumptions of Theorem \ref{Thm-L^p-integrability}, with probability $1$, for all $T>0$,
\begin{equation}\label{s7}
\int_0^T|X(t)|_E^p \, dt <\infty.
\end{equation}
\end{corollary}
\begin{proof}It is enough to consider the case of the process $X_2$. Then, compare with \eqref{int_Y_2}, we have
\begin{equation*}
X_2(t)= \label{int_Y_2'} \int_0^t e^{(t-s)A}\,dY_2(s):=\sum_{\tau_k\leq
t}e^{(t-\tau_k)A}\Delta Y(\tau_k), t\geq 0,
\end{equation*}
where $0<\tau_1<\tau_2<\tau_3<\cdots$
are the jump times of $Y_2$  and $\Delta
Y_2(\tau_k)=\Delta Y(\tau_k)- Y(\tau_k-)$, for all $k$. Note that, c.f. with the proof of  Theorem \ref{Thm-non-cadlag}, that the process $X_2$ is \textit{continuous} as a $U$-valued but as an $E$-valued process, $X_2$ has almost surely, unbounded trajectories.\\
For each $k$ and arbitrary $\omega\in\Omega$,
\begin{eqnarray*}
X_2(t)=S(t-\tau_k)\big[ X_2(\tau_{k}-)+\Delta Y_2(\tau_k)\big],\; t\in [\tau_k,\tau_{k+1}).
\end{eqnarray*}
Note  that $X_2(t)=0$ for $ t\in [0,\tau_{1})$ so that $\int_{0}^{\tau_{1}} |X_2(t)|_E^p\, dt=0<\infty$. If $k\geq 1$, then we have
\begin{eqnarray}
\label{ineq-s8}
\int_{\tau_k}^{\tau_{k+1}} |X_2(t)|_E^p\, dt &\leq & C\int_{0}^{\tau_{k+1}-\tau_k} \vert e^{rA}\vert_{\mathcal{L}(U,E)}^p\, dr \vert X_2(\tau_{k})+\Delta Y_2(\tau_k)\vert_U^p.
\end{eqnarray}
Since the RHS of inequality \eqref{ineq-s8} is $<\infty$ almost surely, the proof is complete.
\end{proof}
\begin{remark}\label{rem-L^p-integrability}
The proof of Corollary \ref{Cor-L^p-integrability} shows that  if $q\in [1,\infty)$ and $\int_{0}^T \vert e^{rA}\vert_{\mathcal{L}(U,E)}^q\, dr<\infty$, then
with probability $1$, for all $T>0$,
\begin{equation}\label{s9}
\int_0^T|X_2(t)|_E^q \, dt <\infty.
\end{equation}
\end{remark}

\subsection{Burgers equation with additive L{\'e}vy noise.}

 The Burgers equation with L{\'e}vy  type noise has
been considered in\cite{Truman}, \cite{Peszat_Z_2007} and recently
in  \cite{Dong}. We will test here applicability of the results
obtained in previous sections. However our results are not directly
comparable with that of \cite{Truman} because we deal with different
form of noise. As we have pointed our earlier the most recent  paper \cite{Dong} considers the case of the compound
Poisson noise and in fact deals with deterministic Burgers equation.

As a preparation  we consider first integrability of the
trajectories of  the O-U process generated by the following Langevin
equation on a one-dimensional domain $(0,1)$, i.e.
 \begin{equation}\label{eqn_Langevin}
\left\{\begin{array}{rcl}
dX(t)&=&\Delta X(t)+dY(t),\; t\geq 0,\\
X(0)&=&0,
\end{array}\right.
\end{equation}
with Dirichlet boundary conditions,  where $Y$ is a L{\'e}vy process defined by \eqref{eqn_subordination_nc} with $\beta\in\Sub(p)$, $p\in (1,2]$ and $W$ being a $H$-cylindrical Wiener process with $H=H^{\theta,2}(0,1)$ for some $\theta\in (0,\frac12)$. We shall prove the following.
\begin{proposition}\label{prop-L^4-integrability}
Under the above assumptions, with probability $1$, for all $T>0$,
\begin{equation}\label{s7}
\int_0^T|X(t)|_{L^4}^4 \, dt <\infty.
\end{equation}
\end{proposition}
\begin{proof}
By Theorem \ref{Theorem_appl_1_ii}, or rather Remark \ref{rem-Theorem_appl_1_ii}, and Corollary \ref{Cor-L^p-integrability}, for each $r<\frac2p-\frac12+\vartheta$ and $s>\frac12-\vartheta$ such that $r+s<\frac2p$, and all $q\in [p,\infty)$,  the process $X$ is $H^{-s,q}(0,1)$-\cadlag and hence is $H^{-s,q}(0,1)$-valued bounded,  takes values in $H^{r,q}(0,1)$ and
\begin{equation*}
\int_0^T|X(t)|_{H^{r,q}}^p \, dt <\infty.
\end{equation*}
Let us choose a particular $s$, i.e. such that $\frac12-\vartheta<s<\frac12$. Put then $r=s(\frac4p-1)$. Then $r+s=\frac{4s}{p}<\frac2p$ and $r<\frac2p-s<\frac2p-\frac12+\vartheta$. This means that the numbers $r$ and $s$ satisfy the assumptions of Theorem \ref{Theorem_appl_1_ii}, Remark \ref{rem-Theorem_appl_1_ii} and Corollary \ref{Cor-L^p-integrability}. Thus
\begin{equation}
\int_0^T|X(t)|_{H^{r,q}}^p \, dt <\infty \;\mbox{ and } \sup  |X(t)|_{H^{-s,q}}<\infty \mbox{ a.s.}
\end{equation}
On the other hand, y equality \eqref{eqn-sobolev_spaces},
 \begin{equation*}
L^{q}(0,1)=H^{0,q}(0,1)=[H^{-s,q}(0,1), H^{r,q}(0,1)]_{\theta},
\end{equation*}
 where
 $\theta \in (0,1)$ satisfy  $\beta=(1-\theta)(-s)+\theta r$, i.e. $\theta =\frac{s}{s+r}$. Hence,
 \begin{equation}\label{ineq-s10}
\vert u\vert_{L^{q}}=\vert u\vert_{H^{-s,q}}^{\frac{r}{r+s}} \vert u\vert_{H^{r,q}}^{\frac{s}{r+s}},\; u\in H^{r,q}(0,1).
\end{equation}
Hence, by the H\"older inequality, with $q=4$, since $\frac{4s}{r+s}=p$, we have
\begin{eqnarray}\nonumber
\int_0^T|X(t)|_{L^4}^4 \, dt &\leq& \int_0^T  \vert X(t)\vert_{H^{-s,q}}^{\frac{4r}{r+s}} \vert X(t)\vert_{H^{r,q}}^{\frac{4s}{r+s}} \, dt\\
&\leq& \big( \sup_{t\in[0,T]} \vert X(t)\vert_{H^{-s,q}}\big)^{\frac{4r}{r+s}} \int_0^T   \vert X(t)\vert_{H^{r,q}}^{p} \, dt <\infty \mbox{ a.s.}
\end{eqnarray}
This concludes the proof.\end{proof}

We pass now to the Burgers equation.  We formulate a result
which, in view of Proposition \ref{prop-L^4-integrability}, can be
proved following the lines of \cite{Brz_2006}. Thus let us
consider equation
\begin{equation}\label{eqn:Burgers}
du+[\mathrm{A}u+B(u)]\,dt=fdt+dY(t), \quad t \geq 0 \end{equation} with the initial
condition
\begin{equation}\label{eqn:Burgers-ic}
u(0)=u_0,
\end{equation}
where  $B(u)=uu_x=\frac12\frac{d}{dx}(u^2)$, $u_0 \in
\mathrm{H}=L^2(0,1)$,
 $\mathrm{V}=H_0^{1,2}(0,1)$ and $f \in
\mathrm{V}^\prime=H^{-1,2}(0,1)$ and $Y$ is a L{\'e}vy process defined by \eqref{eqn_subordination_nc} with $\beta\in\Sub(p)$, $p\in (1,2]$ and $W$ being a $H^{\theta,2}(0,1)$-cylindrical Wiener process  for some $\theta\in (0,\frac12)$.
 Let us notice that $\mathrm{V}$
is a Hilbert space with norm $\Vert u \Vert^2:= \int_0^1 \vert
\nabla u(x)\vert^2\, dx$ which is equivalent with the norm  inherited
from the Sobolev space $H^{1,2}(0,1)$, i.e. $\n{u}\n^2= \int_0^1 \vert
\nabla u(x)\vert^2 dx+\int_0^1 \vert u(x)\vert^2\, dx$.  The above problem
\rf{eqn:Burgers} is of a similar form as the stochastic Navier-Stokes Equations,
however with one essential difference. The nonlinearity $B$ defined
above still satisfies the condition
$$(B(u),u)=0, \;u \in \mathrm{V}$$
but its symmetric bilinear counterpart defined by $B(u,v)=\frac12\frac{d}{dx}(uv)$ no
longer satisfies the stronger condition $(B(u,v),v)=0$ (even for very smooth functions
$u,v$).

%We quote without proof many auxiliary results from \cite{Brz_2006}.

We propose  the following definition of a solution to problem
(\ref{eqn:Burgers})-(\ref{eqn:Burgers-ic}). For this let us fix a number $s$ such that  $\frac12-\vartheta<s<\frac12$.

\begin{definition}\label{def-solution-1}
  An $\mathrm{H}$-valued $(\mathcal{F}_t)_{t \geq 0}$ adapted and  $H^{-s,4}(0,1)$-valued \cadlag     process
$u(t)$, $ t\geq 0$ is a solution to
(\ref{eqn:Burgers})-(\ref{eqn:Burgers-ic}) iff
\begin{equation}
\label{ineq-solution} \sup_{0\le t \leq T} \vert
u(t)\vert_{\mathrm{H}}^2+\int_0^T \vert u(t)\vert^4_{L^4(0,1)}\, dt< \infty \;
\quad \hbox{for each } T>0\; \hbox{a.s.}  \end{equation} and for any
$\psi \in \mathrm{V}\cap H^{2,2}(0,1)$ and $t
>0$,  the following holds a.s.
\begin{eqnarray*}(u(t),\psi)-(u_0,\psi)&-&\int_0^t(u(s),\Delta\psi)\, ds-\frac12
\int _0^t(u^2(s),\nabla\psi)\, ds\\ &=&\int _0^t(f,\nabla\psi)\, ds
+\lb \psi, Y(t)\rb.
\end{eqnarray*}
\end{definition}

\begin{theorem}\label{th:Burgers} For any  $u_0 \in \mathrm{H}$ there exists a unique solution $u(t)$, $t \geq 0$,
 to the stochastic Burgers Equation
(\ref{eqn:Burgers})-(\ref{eqn:Burgers-ic}).
\end{theorem}

Since the construction of the solution given above is pathwise, as a
byproduct of it we can show that there exists a Random Dynamical
System corresponding to the problem
(\ref{eqn:Burgers})-(\ref{eqn:Burgers-ic}). It's properties will not
be however studied in the current publication. One can consult
\cite{Brz_2006} for the Gaussian case.

A tool for studying the existence and uniqueness of solutions for the problem
(\ref{eqn:Burgers}-\ref{eqn:Burgers-ic}) is the following result whose proof can be found in \cite{Brz_2006}. Before we formulate it let denote by
$\mathcal{H}^{1,2}(0,T)$ the space of all functions $v\in L^2(0,T;\mathrm{V}\cap H^{2,2}(0,1)$ such that $v^\prime \in L^2(0,T;\mathrm{V}^\prime)$. Let us recall a well known fact, see e.g. \cite{LM-72-i}, that each element $v\in \mathcal{H}^{1,2}(0,T)$ is equal a.e. to a continuous $H$-valued function defined on the closed interval $[0,T]$ and that the well defined embedding $\mathcal{H}^{1,2}(0,T)\embed C([0,T];H)$ is bounded.

\begin{proposition}\label{prop:Burgers} Assume that $z \in L^4(0,T;L^4(0,1))$, $g \in L^2(0,T;\mathrm{V}^\prime)$ and $v_0
\in \mathrm{H}$. Then, there exists a unique $v \in \mathcal{H}^{1,2}(0,T)$ such that
\begin{eqnarray}\label{eqn:mod-Burgers}
\frac{dv}{dt}&+&Av+B(v,z)+B(z,v)+B(v,v)=g, \quad t \geq 0,
\\
\label{eqn:mod-Burgers-ic} v(0)&=&v_0.
\end{eqnarray}
Moreover, with
\begin{eqnarray*}
K^2 &:=& e^{2 \int_0^T |z(s)|_{L^4}^4\, ds},\; L^2:=|v_0|^2+2\int_0^T|g(s)|_{\mathrm{V}^\prime}^2\, ds,\\
  M^2&:=&\vert v_0\vert^2+9KL
\int_0^T \vert z(t)\vert_{L^4}^4\, dt
 + \int_0^T|g(t)|_{\mathrm{V}^\prime}^2\, dt\\
 N&:=&\vert g \vert_{L^2(0,T;\mathrm{V}^\prime)}+2KLM\vert z
\vert_{L^4(0,T;L^4(0,1))}^2+\frac{T^{1/4}}{\sqrt{2}}K^{3/2}L^{1/2},
\end{eqnarray*}
 we have
\begin{eqnarray*}
\label{ineq:apriori-1}
\sup_{t \in [0,T]} |v(t)|^2 &\leq& K^2L^2,\;
%\nonumber\label{ineq:apriori-2}
\int_0^T \vert \nabla v(t)\vert^2\, dt \leq M^2,
\\
\label{ineq:apriori-3} \int_0^T \vert
v^\prime(t)\vert_{\mathrm{V}^\prime}^2\, dt &\leq& N^2,\;
%\\\label{ineq:apriori-4}
\int_0^T \vert v(t)\vert^4_{L^4(0,1)}\, dt \leq 2T^{1/2}K^3L^3M.
\end{eqnarray*}
 Finally, the map $L^2(0,T;\mathrm{V}^\prime)\times \mathrm{H} \ni (g,v_0) \mapsto v
 \in \mathcal{H}^{1,2}(0,T)$, where $v$ is the unique solution to
 \rf{eqn:mod-Burgers}-\rf{eqn:mod-Burgers-ic}, is real analytic.
 \begin{remark}\label{rem_OP2} {\it Open Problem}. Does there exists a solution to the stochastic Burgers Equation (\ref{eqn:Burgers})-(\ref{eqn:Burgers-ic}) when the   L{\'e}vy process $Y$ is defined by \eqref{eqn_subordination_nc} with $\beta\in\Sub(p)$, $p\in (1,2]$ and $W$ being a $H$-cylindrical Wiener process with $H=H^{0,2}(0,1)=L^2(0,1)$?
\end{remark}
\end{proposition}

\section{Generalizations to non-autonomous equations}

Without much effort some results can be generalized to
the non-autonomous case. . We use some results of Acquistapace
and Terreni  \cite{Acq_T_1987} as presented in  \cite{Veraar_2008p}.

Let us assume that  $(A(t),D(A(t)))_{t\in [0,T]}$ is a family of
closed and densely defined linear operators on a Banach space $E$.

We say that condition \textit{(AT)} is  satisfied if the following
two conditions hold:
\begin{trivlist}
\item[\textit{(AT1)}]  \label{AT1}
There exist constants $\lambda_0\in \mathbb{R}$, $K \ge 0$, and $
\phi \in (\frac{\pi}{2},\pi)$ such that for all $t\in [0,T]$,
$\Sigma(\phi,\lambda_0) \subset \varrho(A(t))$, where $\Sigma(\phi,
{\lambda_0}) = \{{\lambda_0}\}\cup \{\lambda\in \mathbb{C}\setminus
\{{\lambda_0}\}: |\arg (\lambda -{\lambda_0})|\leq \phi\}$
\[ \| R(\lambda, A(t)) \| \le \frac{K}{1+ |\lambda-\lambda_0|},\; \mbox{  for all }
\lambda \in \Sigma(\phi,\lambda_0) .\]
\item[\textit{(AT2)}]  \label{AT2}
There exist constants $L \ge 0$ and $\mu, \nu \in (0,1]$ with $\mu +
\nu >1$ such that for all $\lambda \in \Sigma(\phi,0)$ and $s,t \in
[0,T]$, with $A_{\lambda_0}(t) = A(t) -{\lambda_0}$,
\[ \| A_{\lambda_0}(t)R(\lambda,
A_{\lambda_0}(t))(A_{\lambda_0}(t)^{-1}- A_{\lambda_0}(s)^{-1})\|
\le L |t-s|^\mu (|\lambda|+1)^{-\nu}.  \]
\end{trivlist}
Below we will  denote $\kappa_{\mu,\nu} = \mu+\nu-1$. Note that by
(AT2) $\kappa_{\mu,\nu} \in (0,1]$.

%We will need the following equality for the resolvent operators
%\begin{equation}\label{eq:resodiff}
%\begin{aligned}
%R(\lambda, A(t)) & - R(\lambda, A(s))  \\ & =  A_{\lambda_0}(t)R(\lambda,
%A(t))(A_{\lambda_0}(s)^{-1}- A_{\lambda_0}(t)^{-1}) A_{\lambda_0}(s)R(\lambda, A(s))
%\end{aligned}
%\end{equation}
%for $\lambda\in \rho(A(s))\cap\rho(A(t))$.

If the assumption \textit{(AT1)} is satisfied  and the domains are
constant, i.e.  $D(A(0))=D(A(t))$, $t\in [0,T]$, then the H\"older
continuity of a function $ [0,T] \ni t\mapsto A(t)\in
\cL(D(A(0)),E)$ with exponent $\eta$, implies that the condition
\textit{(AT2)} is satisfied with $\mu = \eta$ and $\nu=1$, see
\cite[Section 7]{Acq_T_1987}. The conditions in that case reduce to
the conditions in the theory of Sobolevskii and Tanabe for constant
domains, cf. \cite{Lunardi_1995, Pazy_1983, Tanabe_1979}.

 The following result is known, see e.g.
\cite[Theorems 6.1-6.4]{Acq_T_1987} and \cite[Theorem
2.1]{Yagi_1991}).
\begin{theorem} \label{thm:exist-parab}
If condition (AT) holds, then there exists a unique strongly
continuous evolution family $(P(t,s))_{0\le s \le t \le T}$ that
solves \eqref{nCP} on $D(A(s))$, i.e.
\begin{equation}\label{nCP}
\begin{aligned}
u'(t) &= A(t) u(t), \ \ t\in [s,T],
\\ u(s) &= x.
\end{aligned}
\end{equation}
 and for all $x\in E$, $P(t, s) x$ is a classical solution of
\eqref{nCP}. Moreover, $(P(t,s))_{0\le s \le t \le T}$ is continuous
on $0 \le s < t \le T$ and there exists a constant $C>0$ such that
for all  $0\leq s<t\leq T$ and all $0 \leq\beta\leq\alpha \le 1$,
\begin{eqnarray}\label{eq:2_13}
\|P(t,s) x\|_{(E, D(A(t)))_{\alpha, 2}} & \le & C(t-s)^{
\beta-\alpha}\|x\|_{(E, D(A(s)))_{\beta, 2}},\; x \in  (E,
D(A(s)))_{\beta, 2},
%\label{eq:2_14} \| P(t,s) - e^{(t-s)A(s)} \| & \le &
%C(t-s)^{\kappa_{\mu,\nu}}.
\end{eqnarray}
where $(E, D(A(s)))_{\beta, 2}$ denotes the real interpolation
space.
\end{theorem}

%Recall from \cite[Theorem 2.2]{Schn} that for all $\alpha\in [0,1]$,
%\begin{equation}\label{eq:strcont}
%\{(t,s): 0\leq s\leq t\leq T\}\ni(t,s)\mapsto (-A_{\lambda_0}(t))^{\alpha}
%P(t,s) (-A_{\lambda_0}(s))^{-\alpha}
%\end{equation}
%is strongly continuous.
%
%Instead, of \eqref{eq:2_13} one could also consider real
%interpolation spaces. Then one may show (cf. \cite[Theorem
%2.2]{Schn}) that for all $p\in [1, \infty]$, $0\leq \beta\leq
%\alpha\leq 1$
%\begin{equation}\label{eq:Prealinter}
%\|P(t,s) x\|_{E_{\alpha,p}^t}\leq C\|x\|_{E_{\alpha,p}^s}, \ x\in
%E_{\alpha,p}^s,
%\end{equation}
%where $C$ depends only on $\alpha, \beta$, $p$ and the constants in
%(AT).

%We also note that by \cite[Theorem 2.3]{Ya} there exists a constant
%$C>0$ such that for all $0\leq s\leq t\leq T$,
%\begin{equation} \label{eq:2_18} \|  ( {\lambda_0}- A(t))^\alpha P(t,s) ( {\lambda_0}-
%A(s))^{-\alpha} - e^{(t-s)A(s)} \|  \le C (t-s)^{\kappa_{\mu,\nu}}
%\end{equation}
%for $\alpha \in (0,1]$. If $\alpha = 0$, one recovers
%\eqref{eq:2_14}. Finally,

 The following result is a generalization of Theorem \ref{Theorem-2} (and of  Remark \ref{rem-Theorem-2} (iv))  to a
non-autonomous case.

We assume that a  triple of spaces $E$, $H$ and $U$ such that the
middle one is a Hilbert space and the other two are separable Banach
spaces is such that condition \eqref{eqn_triple} is satisfied.

We assume that $W=\big(W(t)\big)_{t\geq 0}$ is an $H$-cylindrical
$U$-valued Wiener process. Assume that  $Z=\big(Z(t)\big)_{t\geq 0}$
is a subordinator process  belonging to the class $\Sub(p)$, with
$p\in (1,2]$,  and, as in Theorem \ref{thm-exist-levy},
$Y=\big(Y(t)\big)_{t\geq 0}$ is  an $U$-valued  L{\'e}vy process
defined  by the subordination formula \eqref{eqn_subordination},
i.e.
\begin{equation}\label{eqn_subordination_nc}
Y(t):=W(Z(t)),\; t\geq 0.
\end{equation}
Under the above assumptions we have the following generalization of
Theorem \ref{Theorem-1} (ii).
\begin{theorem}\label{Theorem-3} Assume  that  the conditions \textit{(AT1)-(AT2)} are satisfied and that  the Banach space  $E$ is of type $p$, with $p\in (1,2]$. Assume that $T>0$, $\beta<0<\alpha$ and that $U$ is a separable Banach space such that $U \subset (E, D(A(t)))_{\beta, 2}$, for all $t\in [0,T]$, continuously uniformly in $t$. Assume  also that
\begin{equation}
\label{cond-10} p(\alpha+\beta)<1.
\end{equation} Let $X=\big(X(t)\big)_{t\in [0,T]}$  be the mild solution to the problem
\begin{equation}
\label{eqn_langevin_na} \left\{\begin{array}{rcl}
dX(t)&=&A(t)X(t)\,dt +dY(t),\; t\geq 0,\\
X(0)&=&0,
\end{array}\right.
\end{equation}
given by formula
\begin{equation}\label{eqn_langevin_na_sol}
X(t)=\int_0^tP(t,s)\,dY(s), \; t\in [0,T].
\end{equation}
Then, for each $t\in [0,T]$, $X(t)$ takes a.s. values in the space
$(E, D(A(t)))_{\alpha, 2}$.
\end{theorem}
\begin{proof}
The proof of this result follows the lines of the proof of Theorem \ref{Theorem-1}
(ii) with the formula \eqref{eqn_langevin_na_sol} playing the r\^ole
of the integral $\int_0^t\Psi(s)\, dY(s)$ and the assumption
\eqref{cond-10} together with the inequality \eqref{eq:2_13} being a
replacement of the condition \eqref{cond-1}.
\end{proof}

\end{document}